\documentclass[]{interact}
\usepackage{latexsym}
\usepackage{anyfontsize} %the T&F interact.cls file has some non-integer font sizes (like 5.5) so this fixes that warning
\usepackage{tikz} % for diagrams
\usetikzlibrary{cd} %showing a commutative diagram for demonstration purposes; the author should load whatever they need.
\usepackage{color}
\usepackage{hyperref}
\usepackage{longtable}
\usepackage{url}
\usepackage{breakurl}
\newcommand{\bburl}[1]{\textcolor{blue}{\url{#1}}}
\usepackage{mathtools}
\DeclarePairedDelimiter{\set}{\{}{\}}
\usepackage{amssymb}

\usepackage[labelsep=space]{caption}
\usepackage{placeins}

\makeatletter
\newcommand{\monthyear}[1]{%
	\def\@monthyear{\uppercase{#1}}}
\newcommand{\volnumber}[1]{%
	\def\@volnumber{\uppercase{#1}}}
\makeatother

\newcommand{\ben}{\begin{enumerate}}
	\newcommand{\een}{\end{enumerate}}

\theoremstyle{remark}
\theoremstyle{plain}
\numberwithin{equation}{section} %change this to make globally numbered equations
\newtheorem{thm}{Theorem}[section] %remove [section] to make globally numbered environments
\newtheorem{theorem}[thm]{Theorem}
\newtheorem{lemma}[thm]{Lemma}

\newtheorem{definition}[thm]{Definition}

\newtheorem{corollary}[thm]{Corollary}
\newtheorem{defi}{Definition}
\newtheorem{remark}[thm]{Remark}

\numberwithin{table}{section} %change this and the following line to make globally numbered tables and figures
\numberwithin{figure}{section}

\begin{document}

	\monthyear{Month Year}
	\volnumber{Volume, Number}
	\setcounter{page}{1}
	
	\title{Gaps for PLRS with zero coefficients}

	\author{
		\name{Yinxi Chen\textsuperscript{a}, Zongyun Chen\textsuperscript{b}, Brian Hii\textsuperscript{c} and Steven J. Miller\textsuperscript{d}\thanks{
				Email address: \textsuperscript{a}Casey.Chen\_27@tsinglan.org, \textsuperscript{b}marinachenzongyun@gmail.com, \textsuperscript{c}brianhiififi2020@gmail.com and \textsuperscript{d}sjm1@williams.edu}}
		\affil{
			\textsuperscript{a} 8 Tainan Road, Songshan Lake, Dongguan City, Guangdong Province, China \\
			\textsuperscript{b} 193/198, Soi 9, Koolpunt Ville 7, Moo 4, Muang Chiangmai District, Chiangmai Province, Thailand \\
			\textsuperscript{c} 230 Triana Way, San Ramon, CA 94583, United States \\
			\textsuperscript{d} Department of Mathematics, Williams College, Williamstown, MA 01267}
	}
	
	\thanks{We thank the organizing committee of the $22$nd international Fibonacci conference for the opportunity to present this work and gain valuable feedback.}
	
	\maketitle
	
	{\bf Article type}:  mathematical outreach %change commenting to choose the appropriate article type
	%Article type: mathematical outreach
	%Article type: applications to science
	%Article type: Book Review 
	\bigskip
	
	\begin{abstract}
		Zeckendorf proved that every positive integer has a unique decomposition as a sum of non-adjacent Fibonacci numbers. This result extends to Positive Linear Recurrence Sequences (PLRS), defined by $G_{n+1}=c_1G_n+\cdots+c_LG_{n-L+1},$ where the $c_i$ are non-negative integers and $c_1c_L\ge1$. Bower et al. \cite{BILMT} proved that when every $c_i\ge1$, the limiting gap distribution in generalized Zeckendorf decompositions exhibits geometric decay. When zero coefficients are allowed, however, much less is known; previously, only the Kangaroo case had been treated. Related results follow from the probabilistic framework of Ben-Ari and Miller \cite{BM}, but our goal is to adapt the elementary methods of Bower et al. We study the recurrence $H_{n+1}=H_n+H_{n-p}+H_{n-p-q}.$ For sufficiently large gaps, we prove geometric decay with parameter determined by the largest root of the characteristic polynomial. For smaller gaps, assuming $p<q+2$, we prove that limiting probabilities exist and can be computed using a transfer-matrix method. The case $p\ge q+2$ is not addressed by our small-gap argument.
	\end{abstract}

	
	%%%%%%%%%%%%%%%%%%%%%%%%%%%%%%%%%%%%%%%%%
	
	\section{Introduction}\label{intro}

	\subsection{Background}
	
	The Fibonacci sequence is the sequence $\{F_n\}_{n\geq 1}$ defined by the initial conditions $F_1 = 1, F_2=1,$ and the recurrence relation $F_{n+1}=F_n+F_{n-1}, \text{ for } n\geq 2.$ Thus, each term after the first two is the sum of the two preceding terms, giving $1,1,2,3,5,8,13,\dots$. In Zeckendorf's theorem, the Fibonacci numbers are  $1,2,3,5,8,13,\dots,$ with the repeated $1$ omitted. This convention ensures that every positive integer has a unique representation as a sum of nonconsecutive Fibonacci numbers. In Theorem \ref{generalzeck}, this result is established for \textbf{Positive Linear Recurrence Sequences (PLRS)}, a broader class of sequences that includes the Fibonacci numbers and satisfies a generalized Zeckendorf theorem.
	
	\begin{defi}
		We say a sequence $\{H_n\}_{n=1}^\infty$ of positive integers is a \textbf{Positive Linear Recurrence Sequence (PLRS)} if the following properties hold.
		
		\ben
		
		\item \emph{Recurrence relation:} There are non-negative integers $L, c_1, \dots, c_L$ such that
		\begin{equation}
			H_{n+1} \ = \ c_1 H_n + \dots + c_L H_{n+1-L},
		\end{equation}
		with $L, c_1$ and $c_L$ positive.
		
		\item \emph{Initial conditions:} $H_1 = 1$, and for $1 \le n < L$ we have
		\begin{equation}
			H_{n+1} \ = \
			c_1 H_n + c_2 H_{n-1} + \dots + c_n H_{1}+1.
		\end{equation}
		
		\een
		
		We call a decomposition $\sum_{i=1}^{m} {a_i H_{m+1-i}}$ of a positive integer $N$ (and the sequence $\{a_i\}_{i=1}^{m}$) \textbf{legal} if $a_1>0$, the other $a_i \ge 0$, and one of the following two conditions holds.
		
		\begin{itemize}
			
			\item We have $m<L$ and $a_i=c_i$ for $1\le i\le m$.
			
			\item There exists $s\in\{0, \dots, L\}$ such that
			\begin{equation}\label{scon}
				a_1 \ = \ c_1,\ \ \ a_2 \ = \ c_2,\ \ \ \dots,\ \ \ a_{s-1} \ = \ c_{s-1}\ \ \ {\rm{and}}\  \ \ a_s \ < \ c_s,
			\end{equation} $a_{s+1}, \dots, a_{s+\ell} = 0$ for some $\ell \ge 0$,
			and $\{b_i\}_{i=1}^{m-s-\ell}$ (with $b_i = a_{s+\ell+i}$) is legal.
		\end{itemize}
		
		If $\sum_{i=1}^{m} {a_i H_{m+1-i}}$ is a legal decomposition of $N$, we define the \textbf{number of summands} (of this decomposition of $N$) to be $a_1 + \dots + a_m$.
	\end{defi}
	
	\begin{theorem}[Generalized Zeckendorf's Theorem for PLRS {\cite[Theorem 1.1]{MillerWang}}]\label{generalzeck}
		Every integer $N \in [H_n, H_{n + 1})$ has a unique legal decomposition of the form
		\begin{equation}
			N \ = \ \sum_{i=1}^{n} {a_i H_{n+1-i}}.
		\end{equation}
	\end{theorem}
	
	Uniqueness in decompositions fails for some sequences. For example, in \cite{CFHMN}, Catral et al. found that for some non-PLRS sequences, like Fibonacci quilt sequences, not all integers have exactly one decomposition. 
	Recently, \cite{Salami} studied Zero Linear Recurrence Relations (ZLRRs), another setting in which zero coefficients arise. Unlike our PLRS setting, ZLRRs have a zero first coefficient and are not applicable as our zero coefficients are in a different location, which can lead to non-unique decompositions and complications for the greedy algorithm.\footnote{After choosing the largest available summand, the algorithm continues by choosing smaller-indexed summands.}
	
	The uniqueness in decompositions for PLRS makes it natural to ask statistical questions about these decompositions. Since all legal decompositions of PLRS sequences are also the ones obtained by the greedy algorithm, studying all integers $N$ in an interval such as \([H_n,H_{n+1})\) and asking what a typical legal decomposition looks like is reasonable as now all integers should have the same largest summand and are thus of the same scale.
	
	\begin{remark}
		We are localizing so that the objects studied are of similar scale. Having $N$ in $[H_n, H_{n+1})$ ensures the largest summand of its decomposition is $H_n$; thus all integers in this interval have the same largest summand. While it is appropriate to so normalize, we would see similar behavior if instead we studied all integers from $[1, H_{n+1})$ because that interval is the union of the intervals $[H_i, H_{i+1})$. As $\{H_i\}$ is an exponentially growing sequence, as $n\to\infty$ almost all of the integers in $[1, H_{n+1})$ are in any union of intervals over indices from $n-\lfloor \log n\rfloor$ to $n$, and changing the interval by such a small amount will not change to first order any of the quantities we study.
	\end{remark}   
	%Otherwise we could be investigating $N_1$ with largest summand say $H_{19}$ and $N_2$ with largest summand $H_{10^{10}}$, so much larger that we would expect remarkably different behavior. One could look at all numbers up to $H_n$ or $H_{n+1}$ and find that, as the behavior changes very slowly, if you look at intervals $I_k := [H_k, H_{k+1})$ then if we look at $\bigcup_{k=n-\lfloor \log n \rfloor}^{n} I_k$, then each of the $I_k$ has about the same behavior, and almost all of the integers at most $H_{n+1}$ live in these $log(n)$ intervals, so going all the way down to $1$ won't change anything.
	
	%PROF MILLER: PLEASE EXPAND
	%\end{remark}
	
	Lekkerkerker proved that the average number of summands in Zeckendorf decompositions of integers in \([F_n,F_{n+1})\) grows linearly in \(n\) (see Theorem \ref{generallek}). Later, various works studied limiting laws for the number of summands. For example, Kolo\u{g}lu, Kopp, Miller, and Wang \cite{KologluKoppMillerWang} proved that the distribution of the number of summands in Zeckendorf decompositions converges to a Gaussian as $n \to \infty$. Miller and Wang~\cite{MillerWang} extended this result to generalized Zeckendorf decompositions arising from PLRS.

	\begin{definition}
		We say that \(f(x)=O(g(x))\) as \(x \to \infty\) if there exist constants
		\(C>0\) and \(x_0\) such that
		\begin{equation}
			|f(x)| \ \leq \ Cg(x)
		\end{equation}
		for all \(x \geq x_0\).
		
		We say that \(f(x) = o(g(x))\) as \(x \to\infty\) if
		\begin{equation}
			\lim_{x\to\infty}\frac{f(x)}{g(x)} \ = \ 0.
		\end{equation}
	\end{definition}
	
	\begin{theorem}[Generalized Lekkerkerker’s Theorem for PLRS {\cite[Theorem 1.2]{MillerWang}}]\label{generallek}
		Let $\{H_n\}_{n=1}^{\infty}$ be a Positive Linear Recurrence Sequence. Choose $N$ uniformly at random from $[H_n, H_{n + 1})$, and let $K_n$ denote the number of summands in the legal decomposition of $N$. Denote the mean of $K_n$ by $\mu_n$. Then there exist constants $C > 0$, $d$, and $\gamma_1 \in (0,1)$ depending only on $L$ and the $c_i$’s in the recurrence relation of the $H_n$’s such that
		\begin{equation}
			\mu_n \ = \ Cn + d + o(\gamma_1^n).
		\end{equation}
	\end{theorem}
	
	A related problem, which is the direction of this paper, is the distribution of gaps between summands. If a legal decomposition contains only the summands \(H_{j_1},H_{j_2},\dots,H_{j_k}\), with the indices written in non-increasing order, then the differences between consecutive indices are called \textbf{gap sizes} (i.e., $j_1 - j_2$ where $j_1 \ge j_2$ represents the size of the gap between $H_{j_1}$ and $H_{j_2}$). Previous work studied gap distributions for PLRS when all recurrence coefficients are positive \cite{BILMT}, Kangaroo recurrences \cite{BBGILMT}, which are PLRS with zero coefficients that occur in groups of size $g$, and ZLRR \cite{Salami}, where the first recurrence coefficient is zero. Ben-Ari and Miller \cite{BM} also developed a probabilistic approach to generalized Zeckendorf decompositions that applies to certain recurrences with zero coefficients. Their framework establishes general limiting results but the techniques require a lot more heavy machinery --- thus the value of our paper is not decreased. 
	Although in each of the these settings, the limiting gap distribution was shown to exhibit geometric decay, their results complement rather than cover the PLRS case studied here, as our zero coefficients occur in different locations.
	
	The purpose of this paper is to test the generality of the geometric decay. We study the simplest next case --- three-term sequences with coefficients 0 and 1, but unlike the Kangaroo recurrence, we can have coefficients 0 and 1 placed anywhere as long as the resulting sequence is still a PLRS. In particular, we analyze the sequences of form
	\begin{equation}
		H_{n + 1} \ = \ H_n + H_{n - p} + H_{n - p - q},
	\end{equation}
	with initial conditions to match those for a PLRS. We consider all integers $N$ in the interval $[H_n, H_{n + 1})$.
	Our goal is to determine whether, as $n \to \infty$, the limiting gap distribution still geometrically decays by a factor of $r_1$ for these more general sequences. We prove this geometric decay for all sufficiently large gaps. For the finitely many smaller gaps, when $p<q+2$, we prove that the limiting probabilities exist and give a transfer-matrix formula for computing them. It is very plausible that the existence of the small gap limit can still be proved when $p \ge q+2$ but our current transfer matrix does not prove it.
	
	Before we state our main theorems, we need a definition.
	
	\begin{definition}
		Define $P_n(g)$ as the probability density of a gap of size $g$ from all decompositions of integers in $[H_n, H_{n + 1})$, and set
		\begin{equation}
			P(g) \ := \ \lim_{n \to \infty} P_n(g).
		\end{equation}
	\end{definition}
	
	\subsection{Main Theorems}
	
	\begin{theorem}\label{gap_large}
		Consider the recurrence $H_{n+1}=H_n+H_{n-p}+H_{n-p-q}$, with initial conditions to match a PLRS.
		For \textbf{large gap sizes} (i.e., $g > \max\{p, q\}$), $P(g)$ decays geometrically in $r_1$ --- namely,
		\begin{equation}
			P(g) 
			\ = \ K_{p,q} \cdot r_1^{-g} 
			\ = \ \frac{\kappa_1}{C}(r_1^q+1)^2 r_1^{-2p-2q} \cdot r_1^{-g},
		\end{equation} with $r_1$ the dominant root of the characteristic polynomial, $\kappa_{1}$ the generalized Binet constant in front of $r_1$, and $C$ the computable Lekkerkerker constant.
	\end{theorem}
	
	\begin{theorem}\label{gap_small}
		Consider the recurrence $H_{n+1}=H_n+H_{n-p}+H_{n-p-q}$, with initial conditions to match a PLRS.
		Define \textbf{small gap sizes} as $1 \le g \le \max\{p, q\}$.
		For each fixed choice of $p$, $q$, and small gap size $g$, the limiting probability $P(g)$ exists when $p < q + 2$ and its value generally depends on the ordered pair $(p, q)$ and on $g$; namely,
		
		\begin{equation}
			P(g) 
			\ = \ K_{p, q, g} 
			\ = \ \frac{r_1^{- g - 1} \mathbf{1}^T P M_1 M_0^{g - 1} M_1 P v_{p + q}}{C \kappa_{1} (r_1^{q} + 1)},
		\end{equation} with $r_1$ the dominant root of the characteristic polynomial, $\kappa_{1}$ the generalized Binet constant in front of $r_1$, $C$ the computable Lekkerkerker constant, and $P$ the computable Perron projection matrix. The matrices $M_0$ and $M_1$ are the transfer matrices corresponding respectively to appending a $0$ and appending a
		$1$ to the current state, and $v_{p+q}$ is the
		initial state-count vector for legal strings of length $p+q$ with the largest summand $H_n$ forced.
	\end{theorem}
	
	Tables \ref{tab:large-gap-data} and \ref{tab:small-gap-data} present numerical values of $P(g)$ for several choices of $p, q,$ and $g$.
	Table \ref{tab:large-gap-data} summarizes the numerical results for the large-gap probabilities. For each ordered pair $(p, q)$ and each large gap size $g > \max \{p, q\}$, we list the limiting probabilities predicted by Theorem \ref{gap_large}, and when $n = 35$ the finite-$n$ approximation $P_{35}(g)$ obtained by brute-force counting. The two values are in close agreement, providing numerical support for the large-gap formula.
	Table \ref{tab:small-gap-data} summarizes the numerical results for the small-gap probabilities in the cases where $p<q+2$. For each pair $(p,q)$ and each small gap size $g$, we list the dominant root $r_1$ (or the Perron eigenvalue $\lambda$ of the transfer matrix), the limiting probability predicted by Theorem \ref{gap_small}, the total contribution from the large-gap range, the sum over all gap sizes, and the finite-$n$ approximation $P_{35}(g)$. The limiting small-gap probabilities are in close agreement with the corresponding values of $P_{35}(g)$ obtained by brute-force counting. Together with the fact that the total probabilities are close to $1$ with only rounding errors, this provides a numerical check of the transfer-matrix method.
	
	\begin{remark}
		The numerical data suggest that the small-gap limiting probabilities may exhibit geometric decay over certain ranges, although this behavior does not necessarily hold across all small gaps. The exceptional behavior occurs at particular boundary gap sizes, where the recurrence may create additional interactions between the summands on the two sides of the gap. A precise description of the geometric decay among the small gaps is given in Theorem \ref{thm:small_gap_decay}, together with its proof.
	\end{remark}
	
	\begin{table}[ht]
		\centering
		\scriptsize
		\begin{tabular}{c c c c c c c c c}
			\hline
			$p$ & $q$ & $g$ & $\kappa_1$ & $C_{\mathrm{Lek}}$ & $r_1$
			& $P(g)$
			& $\sum P(g)$ large $g$
			& $P_{35}(g)$ \\
			\hline
			
			1 & 2 & 3 & 0.722124418 & 0.326378645 & 1.754877666 & 0.233293644 & 0.542341925 & 0.230812321 \\
			&   & 4 &             &             &             & 0.132940118 &             & 0.127630879 \\
			&   & 5 &             &             &             & 0.075754636 &             & 0.070509529 \\
			\hline
			
			2 & 1 & 3 & 0.723606798 & 0.276393202 & 1.618033989 & 0.236067978 & 0.618033989 & 0.240416291 \\
			&   & 4 &             &             &             & 0.145898034 &             & 0.144237126 \\
			&   & 5 &             &             &             & 0.090169944 &             & 0.086456041 \\
			\hline
			
			1 & 3 & 4 & 0.814040668 & 0.297756511 & 1.704902776 & 0.160784731 & 0.388879636 & 0.154548036 \\
			&   & 5 &             &             &             & 0.094307273 &             & 0.087898873 \\
			&   & 6 &             &             &             & 0.055315338 &             & 0.049943357 \\
			\hline
			
			2 & 2 & 3 & 0.811108777 & 0.241709338 & 1.570147312 & 0.281798472 & 0.776054020 & 0.289167104 \\
			&   & 4 &             &             &             & 0.179472633 &             & 0.178824829 \\
			&   & 5 &             &             &             & 0.114303054 &             & 0.110489019 \\
			\hline
			
			3 & 1 & 4 & 0.819222321 & 0.155953755 & 1.497094049 & 0.258396563 & 0.778210798 & 0.190569410 \\
			&   & 5 &             &             &             & 0.172598751 &             & 0.124254447 \\
			&   & 6 &             &             &             & 0.115289184 &             & 0.098140515 \\
			\hline
			
			1 & 4 & 5 & 0.894229519 & 0.282740362 & 1.673648546 & 0.109296593 & 0.271542312 & 0.101682253 \\
			&   & 6 &             &             &             & 0.065304387 &             & 0.058859127 \\
			&   & 7 &             &             &             & 0.039019176 &             & 0.034035468 \\
			\hline
			
			2 & 3 & 4 & 0.891359789 & 0.220855136 & 1.538496592 & 0.208893679 & 0.596813829 & 0.208780412 \\
			&   & 5 &             &             &             & 0.135777798 &             & 0.131677290 \\
			&   & 6 &             &             &             & 0.088253558 &             & 0.082971234 \\
			\hline
			
			3 & 2 & 4 & 0.896185072 & 0.194254004 & 1.465571232 & 0.216756572 & 0.682327803 & 0.224258782 \\
			&   & 5 &             &             &             & 0.147899036 &             & 0.148558566 \\
			&   & 6 &             &             &             & 0.100915624 &             & 0.098322866 \\
			\hline
			
			4 & 1 & 5 & 0.907738526 & 0.128265612 & 1.419632763 & 0.216127106 & 0.731165790 & 0.175694096 \\
			&   & 6 &             &             &             & 0.152241560 &             & 0.140324066 \\
			&   & 7 &             &             &             & 0.107240100 &             & 0.098345708 \\
			\hline
			
			1 & 5 & 6 & 0.962099524 & 0.275311854 & 1.653634587 & 0.073017650 & 0.184727850 & 0.065584710 \\
			&   & 7 &             &             &             & 0.044155856 &             & 0.038383830 \\
			&   & 8 &             &             &             & 0.026702305 &             & 0.022439490 \\
			\hline
			
			2 & 4 & 5 & 0.964324423 & 0.208037272 & 1.516784253 & 0.154198981 & 0.452580714 & 0.149506649 \\
			&   & 6 &             &             &             & 0.101661776 &             & 0.095565703 \\
			&   & 7 &             &             &             & 0.067024546 &             & 0.061025875 \\
			\hline
			
			3 & 3 & 4 & 0.968233828 & 0.178392975 & 1.443268791 & 0.245780788 & 0.800254264 & 0.255440245 \\
			&   & 5 &             &             &             & 0.170294535 &             & 0.171864302 \\
			&   & 6 &             &             &             & 0.117992252 &             & 0.115533807 \\
			\hline
			
			4 & 2 & 5 & 0.977162977 & 0.151370617 & 1.397035275 & 0.191225916 & 0.672860492 & 0.184649969 \\
			&   & 6 &             &             &             & 0.136879805 &             & 0.129954309 \\
			&   & 7 &             &             &             & 0.097978775 &             & 0.090410099 \\
			\hline
			
			5 & 1 & 6 & 0.990756210 & 0.139847372 & 1.365254707 & 0.145956029 & 0.545556710 & 0.170687658 \\
			&   & 7 &             &             &             & 0.106907545 &             & 0.123037779 \\
			&   & 8 &             &             &             & 0.078305934 &             & 0.087825712 \\
			\hline
			
		\end{tabular}
		\caption{Large-gap data.}
		\label{tab:large-gap-data}
	\end{table}

	\begin{table}[ht]
		\centering
		\scriptsize
		\begin{tabular}{c c c c c c c c}
			\hline
			$p$ & $q$ & $g$ & $r_1$ (or $\lambda$)
			& $P(g)$
			& $\sum P(g)$ large $g$
			& $\sum P(g)$ all $g$
			& $P_{35}(g)$ \\
			\hline
			
			1 & 2 & 1 & 1.754877666 & 0.196866065 & 0.542341925 & 1.000000001 & 0.216044289 \\
			&   & 2 &             & 0.260792011 &             &             & 0.268433483 \\
			\hline
			
			2 & 1 & 1 & 1.618033989 & 0           & 0.618033989 & 1.000000000 & 0 \\
			&   & 2 &             & 0.381966011 &             &             & 0.400385763 \\
			\hline
			
			1 & 3 & 1 & 1.704902776 & 0.143768434 & 0.388879636 & 1.000000000 & 0.165810625 \\
			&   & 2 &             & 0.294572363 &             &             & 0.303503501 \\
			&   & 3 &             & 0.172779567 &             &             & 0.172979258 \\
			\hline
			
			2 & 2 & 1 & 1.570147312 & 0           & 0.776054020 & 1.000000000 & 0 \\
			&   & 2 &             & 0.223945980 &             &             & 0.241540540 \\
			\hline
			
			1 & 4 & 1 & 1.673648546 & 0.101562455 & 0.271542312 & 1.000000000 & 0.126665506 \\
			&   & 2 &             & 0.320744607 &             &             & 0.329814431 \\
			&   & 3 &             & 0.191643943 &             &             & 0.191499920 \\
			&   & 4 &             & 0.114506683 &             &             & 0.111096625 \\
			\hline
			
			2 & 3 & 1 & 1.538496592 & 0           & 0.596813829 & 1.000000000 & 0 \\
			&   & 2 &             & 0.177255371 &             &             & 0.196966720 \\
			&   & 3 &             & 0.225930800 &             &             & 0.235126066 \\
			\hline
			
			3 & 2 & 1 & 1.465571232 & 0           & 0.682327803 & 0.999999999 & 0 \\
			&   & 2 &             & 0           &             &             & 0 \\
			&   & 3 &             & 0.317672196 &             &             & 0.335550315 \\
			\hline
			
			1 & 5 & 1 & 1.653634587 & 0.069613145 & 0.184727850 & 1.000000000 & 0.097780841 \\
			&   & 2 &             & 0.340239296 &             &             & 0.348675306 \\
			&   & 3 &             & 0.205752407 &             &             & 0.204902952 \\
			&   & 4 &             & 0.124424349 &             &             & 0.120311951 \\
			&   & 5 &             & 0.075242953 &             &             & 0.070579825 \\
			\hline
			
			2 & 4 & 1 & 1.516784253 & 0           & 0.452580714 & 1.000000001 & 0 \\
			&   & 2 &             & 0.137119308 &             &             & 0.159214967 \\
			&   & 3 &             & 0.247274492 &             &             & 0.257301422 \\
			&   & 4 &             & 0.163025487 &             &             & 0.165440707 \\
			\hline
			
			3 & 3 & 1 & 1.443268791 & 0           & 0.800254264 & 1.000000000 & 0 \\
			&   & 2 &             & 0           &             &             & 0 \\
			&   & 3 &             & 0.199745736 &             &             & 0.217320999 \\
			\hline
			
		\end{tabular}
		\caption{Small-gap data.}
		\label{tab:small-gap-data}
	\end{table}

	%%%%%%%%%%%%%%%%%%%%%%%%%%%%%%%%%%%%%%%%%%%%%%%%%%%%%%%%%%%%%%%%%%%%%%%%%%%%%%%%%%%%%%%%%%%%%%%%%%%%%%%%%%%%%%%%%%%%%%%%%%%%%%%%%%%%%%%%%%%%%%%%%%%%%%%%%%%%%%%%%%%%%%
	
	\section{Proof of Theorem \ref{gap_large}} \label{large gap section}
	
	\subsection{Total number of gaps}
	
	By Theorem \ref{generallek}, the average number of summands in the legal decompositions of all $N \in [H_n, H_{n + 1})$ is $Cn + d + o(\gamma_1^n)$. Since a decomposition with $k$ summands contains $k - 1$ gaps, the average number of gaps is $Cn + d - 1 + o(\gamma_1^n).$
	
	The total number of gaps is equal to the average number of gaps multiplied by the total number of decompositions. By the uniqueness of legal decompositions, the total number of decompositions in $[H_n, H_{n + 1})$ equals the number of integers in this interval, namely $H_{n + 1} - H_n = H_{n - p} + H_{n - p - q}$. Thus, the total number of gaps in the legal decompositions of all $N \in [H_n, H_{n + 1})$ is
	\begin{equation}
		(H_{n - p} + H_{n - p - q}) (Cn + d - 1 + o(\gamma_1^n)).
	\end{equation}
	
	\subsection{Number of occurrences of a gap of size $g$}
	
	\begin{lemma}\label{gap_no.}
		We define $X_{i, g}(n)$ as the total number of legal decompositions where all summands are chosen from $\set{H_1, \dots, H_n}$, such that $H_{n}$, $H_{i + g}$ and $H_i$ are forcefully chosen, all indices between $H_i$ and $H_{i + g}$ cannot be chosen, and all remaining terms are free to be chosen subject to the legality conditions.
		The number of occurrences of a gap of size $g > \max\{p, q\}$ in the legal decompositions of all $N \in [H_n, H_{n + 1})$ is
		\begin{equation}
			\sum_{i = 1}^{n - g} X_{i, g}(n) \ = \ \sum_{i = 1}^{n - g} (H_{i - p} + H_{i - p - q})(H_{n - i - g - 2p} + 2 H_{n - i - g - 2p - q} + H_{n - i - g - 2p - 2q}).
		\end{equation}
	\end{lemma}
	
	\begin{proof}
		We prove Lemma \ref{gap_no.} by counting the number of possible arrangements to the left and to the right\footnote{Throughout this paper, we list the indices in decreasing order, as produced by the greedy algorithm. Thus the ``right'' side of a gap refers to the smaller-indexed summands.} of a gap of size $g$. Let $H_i$ and $H_{i + g}$ be consecutive summands in the decomposition. Now we discuss the number of arrangements to either side of the gap $g$ (see Figure \ref{fig:NumberGap}).
		
		\begin{figure}[h]
			\begin{center}
				\scalebox{.21}{\includegraphics{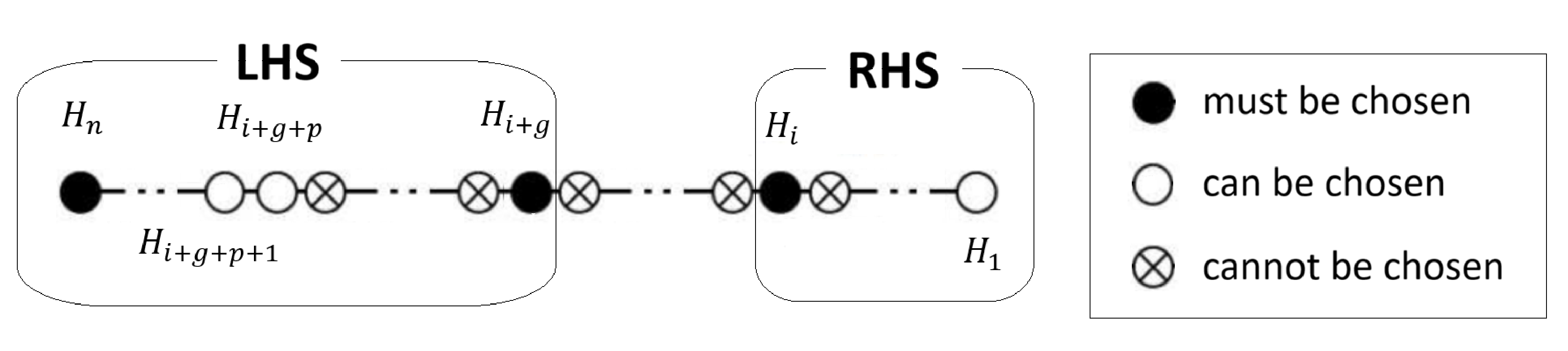}}
				\caption{\label{fig:NumberGap} Number of gaps of size $g$.}
		\end{center}\end{figure}
		
		\begin{remark}
			Lemma \ref{gap_no.} can only be proven with large $g$, since, otherwise, the left and right sides of the gap from $H_i$ to $H_{i + g}$ are not independent and can interact with each other through the recurrence relation.
		\end{remark}
		
		First, the number of arrangements to the right of the gap (including $H_i$) equals the number of legal choices from $\set{H_1, H_2, \dots, H_i}$ with $H_i$ chosen. By the existence and uniqueness of legal decompositions for this PLRS, these choices are in bijection with the total number of integers in the interval $[H_i, H_{i + 1})$. Thus, the number of such arrangements is $H_{i + 1} - H_i = H_{i - p} + H_{i - p - q}$.
		
		Next, the number of arrangements to the left of the gap (including $H_{i + g}$) correspond to the number of legal choices from $\set{H_{i + g}, H_{i + g + 1}, \dots, H_n}$ with $H_{i + g}$ and $H_n$ chosen. By shifting indices, this bijects with the number of legal choices from $\set{H_1, H_2, \dots, H_{n - i - g + 1}}$ in which $H_1$ and $H_{n - i - g + 1}$ must be chosen. In order to compute this quantity, we find the number of arrangements where $H_{n - i - g + 1}$ must be used and subtract the number of arrangements where $H_{n - i - g + 1}$ must be used but $H_1$ cannot be used.
		
		The number of arrangements where $H_{n - i - g + 1}$ must be used is $H_{n - i - g + 2} - H_{n - i - g + 1} = H_{n - i - g - p + 1} + H_{n - i - g - p - q + 1}$, and the number of arrangements where $H_{n - i - g + 1}$ must be used but $H_1$ cannot be used is $H_{n - i - g + 1} - H_{n - i - g} = H_{n - i - g - p} + H_{n - i - g - p - q}$.
		Thus, the number of valid left-hand arrangements is $(H_{n - i - g - p + 1} + H_{n - i - g - p - q + 1}) - (H_{n - i - g - p} + H_{n - i - g - p - q}) = H_{n - i - g - 2p} + 2 H_{n - i - g - 2p - q} + H_{n - i - g - 2p - 2q}$.
		
		Multiplying the number of choices on the left- and right-hand-sides and summing over all possible positions $i$ of the gaps gives the total number of occurrences of a gap of size $g$:
		\begin{equation}\label{no.gap_g}
			\sum_{i = 1}^{n - g} X_{i, g}(n) \ = \ \sum_{i = 1}^{n - g} (H_{i - p} + H_{i - p - q})(H_{n - i - g - 2p} + 2 H_{n - i - g - 2p - q} + H_{n - i - g - 2p - 2q}).
		\end{equation}
		In cases where negative indices\footnote{The negative-index extension is used only to handle boundary terms uniformly. Since the backward values are uniquely determined by the same recurrence, it only gives a consistent way to evaluate the summation near the endpoints without introducing any new legal decompositions.} appear, we extend the recurrence backward until $-2p - 2q$.
	\end{proof}
	
	\subsection{The distribution of the average gap measure}
	
	The probability of choosing a gap of size $g$ from all decompositions of all $N \in [H_n, H_{n + 1})$ is
	\begin{equation}\label{png}
		P_n(g) \ = \ \frac{\sum_{i = 1}^{n - g} X_{i, g}(n)}{(H_{n - p} + H_{n - p - q}) (Cn + d - 1 + o(\gamma_1^n))}.
	\end{equation}
	By the generalized Binet formula \cite{MartinezMillerMizgerdMurphySun2},
	\begin{equation}
		H_n \ = \ \kappa_1 r_1^n + O(\rho ^n),
	\end{equation}
	where $r_1$ is the unique dominant root and $|r_2|, |r_3|, \dots ,|r_{p + q + 1}| < r_1$. We can prove that $\kappa_1$ must be positive. We know that $H_n$ and $r_1$ are positive and we have
	\begin{equation}
		\frac{H_n}{r_1^n}
		\ = \ \kappa_1 + O \left(\left(\frac{\rho}{r_1}\right)^n\right).
	\end{equation}
	The error term tends to zero, as $|\rho|<r_1$.
	
	We can prove the uniqueness of the leading root by proving that it is  not of multiplicity $2$ or greater. We prove this fact for all PLRS instead of just our specific three-term one to aid future researchers. To do so, it suffices to prove that there are no roots of multiplicity $2$. If there is a root $x$ of multiplicity $2$, then for some $x$ we have $f(x) = f'(x) = 0$. Suppose there exists such a root. The characteristic equation is
	\begin{equation}
		f(x) \ = \ x^{L} - \sum_{j = 1}^{L} c_j x^{L - j}.
	\end{equation}
	The derivative of the characteristic equation is 
	\begin{equation}
		f'(x) \ = \ L x^{L - 1} - \sum_{j = 1}^{L} c_j (L - j) x^{L - j - 1}.
	\end{equation}
	Calculating the roots of $f'(x) = 0$, we can easily see $x = 0$ is a root and now we prove that there does not exist any positive roots which can satisfy $f'(x) = 0$. To do so, we multiply both sides of $f'(x) = 0$ by $x$ to obtain
	\begin{equation}
		x f'(x) \ = \ L x^L - \sum_{j = 1}^{L} c_j (L - j) x^{L - j} \ = \ 0.
	\end{equation}
	From the characteristic equation $f(x) = 0$, we obtain $x^{L} = \sum_{j = 1}^{L} c_j x^{L - j}$ for roots $x$ for $f$.
	Substituting this into the previous equation, we get
	\begin{equation}
		x f'(x) \ = \ L \sum_{j = 1}^{L} c_j x^{L - j} - \sum_{j = 1}^{L} c_j (L - j) x^{L - j} \ = \ 0.
	\end{equation}
	Simplifying yields
	\begin{equation}
		x f'(x) \ = \ \sum_{j = 1}^{L} j c_j x^{L - j} \ = \ 0,
	\end{equation}
	which cannot be true if $x$ is positive as each $c_j \ge 0$ (and at least one is positive).
	Thus, there are no positive roots of $f'(x) = 0$, which means that there cannot be a double leading root for any PLRS.

	Let $ a = i - p$, $ b = i - p - q$, $ c = n - i - g - 2p$, $ d = n - i - g - 2p - q$ and $ e = n - i - g - 2p - 2q$ for notation convenience.
	From Lemma \ref{gap_no.}, we have
	\begin{equation} \label{pq_expansion}
		X_{i, g} (n)
		\ = \ (H_a + H_b)(H_c + 2H_d + H_e).
	\end{equation}
	We now separate the summation into three regions:
	\begin{equation}
		\sum_{i=1}^{n-g} X_{i, g}(n) \ = \ \sum_{i=1}^{f(n)} X_{i, g}(n) + \sum_{i=f(n)+1}^{n-g-f(n)-1} X_{i, g}(n) + \sum_{i=n-g-f(n)}^{n-g} X_{i, g}(n),
	\end{equation}
	where we let $f(n) = \log n$ such that\footnote{We choose $\log n$ for definiteness, but we just need $f(n) \to \infty$ and $f(n)/n \to 0$, when $n \to \infty$.}, when $n \to \infty$, we have $f(n) \to \infty$. Let $h(n) = n$. By L'H\^{o}pital's Rule, we get
	\begin{equation}
		\lim_{n \to \infty} \frac{f(n)}{h(n)}
		\ = \ \lim_{n \to \infty} \frac{f'(n)}{h'(n)}
		\ = \ \lim_{n \to \infty} \frac{1/n}{1}
		\ = \ 0 
	\end{equation}
	(i.e., $f(n) = o(n)$).
	
	The first and third sums correspond to the left and right edge regions respectively, each of which contains at most $f(n)$ terms.
	Since each non-leading product has order at most $r_1^{n-g-3p}$, we get the total left edge size as
	\begin{equation}
		\sum_{i=1}^{f(n)} X_{i, g}(n)  \ = \ O(f(n) \cdot r_1^{n - g - 3p}).
	\end{equation}
	As for the denominator of \eqref{png}, it is of order $n \cdot r_1^{n - p}$.
	
	By forming the ratio, we obtain that 
	\begin{equation}
		\frac{\text{left edge contribution}}{\text{denominator}}
		\ = \ O \left( \frac{f(n) \cdot r_1^{n - g - 3p}}{n \cdot r_1^{n-p}} \right)
		\ = \ O \left( \frac{f(n)}{n} \cdot r_1^{- g - 2p}\right).
	\end{equation}
	As $f(n)/n \to 0$, the left edge contribution is negligible. Similarly, we find that the right edge contribution is also negligible.
	
	We now consider the middle range 
	\begin{equation}
		f(n) +1 \ \leq \ i \ \leq \ n-g-f(n) - 1.
	\end{equation}
	In this region, all of $a, b, c, d, e \ge f(n) + O(1)$.
	So the Binet approximation in the middle range is of form
	\begin{equation}
		H_m \ = \ \kappa_1 r_1^m + O(\rho ^m) 
		\ = \ \kappa_1 r_1^m \left( 1 + O\left( \left( \frac{\rho}{r_1} \right)  ^m\right) \right) 
		\ = \ \ \kappa_1 r_1^m (1+o(1)),
	\end{equation}
	where $\kappa_{1} > 0$ and $|\rho|<r_1$ implying $0<|\rho/r_1|<1$.
	
	Note that $a + c = n - g - 3p$, $a + d = b + c = n - g - 3p - q$, $a + e = b + d = n - g - 3p - 2q$, and $b + e = n - g - 3p - 3q$.
	Therefore, from \eqref{pq_expansion} we have
	\begin{align}
		X_{i,g}(n)
		&\ = \ (H_a+H_b)(H_c+2H_d+H_e) \notag\\
		&\ = \ \kappa_1^2(r_1^a+r_1^b)(r_1^c+2r_1^d+r_1^e)(1+o(1)) \notag\\
		&\ = \ \kappa_1^2 r_1^{n-g-3p-3q}(r_1^q+1)^3
		+o\!\left(r_1^{n-g-3p-3q}\right).
		\label{eq:Xig-asymptotic}
	\end{align}
	for $i$ in the middle range.
	Since $2f(n) + O(1) = o(n)$, summing the $n-2f(n)+O(1)$ terms over the middle range gives
	\begin{equation}\label{num_asym}
		\begin{aligned}
			\sum_{i=f(n)+1}^{n-g-f(n)-1} X_{i, g}(n)
			& \ = \ \kappa_1^2 r_1^{n - g - 3p - 3q} (r_1^q + 1)^3 (n - o(n))  + o(n r_1^{n - g - 3p - 3q}) \\
			& \ = \ \kappa_1^2 r_1^{n-g-3p - 3q} (r_1^{q} + 1)^3 n + o(n r_1^{n-g-3p - 3q}).
		\end{aligned}
	\end{equation}
	Substituting this into \eqref{png}, and using
	$ H_{n-p} + H_{n - p - q} = \kappa_1 (1 + r_1^{q}) r_1^{n-p - q} + o(r_1^{n-p}), $
	we obtain
	\begin{equation}
		P_n(g)
		\ = \ \frac{\kappa_1^2 r_1^{n-g-3p - 3q} (r_1^{q} + 1)^3 n + o(n r_1^{n-g-3p - 3q})}
		{(\kappa_1 (r_1^{q} + 1) r_1^{n-p - q} + o(r_1^{n-p})) (Cn + d - 1 + o(\gamma_1^n))}.
	\end{equation}
	The denominator can be written as
	$$
	(\kappa_1 (r_1^{q} + 1) r_1^{n-p - q} + o(r_1^{n-p})) (Cn + O(1)) \ \sim \ (\kappa_1 (r_1^{q} + 1) r_1^{n-p - q} + o(r_1^{n-p})) Cn,
	$$
	so
	\begin{equation}
		P_n(g)
		\ = \ \frac{\kappa_1^2 r_1^{n-g-3p - 3q} (r_1^{q} + 1)^3 n + o(n r_1^{n-g-3p - 3q})} {\kappa_1 (r_1^{q} + 1) r_1^{n-p - q} Cn + o(n r_1^{n-p})}.
	\end{equation}
	Simplifying gives us
	\begin{equation}
		P_n(g)
		\ = \ \frac{\kappa_1}{C} (r_1^{q} + 1)^2 r_1^{-2p - 2q} \cdot r_1^{-g} + o(1).
	\end{equation}
	Taking the limit as $n \to \infty$, we obtain
	\begin{equation}
		P(g) 
		\ = \ \lim_{n \to \infty} P_n(g) 
		\ = \ \frac{\kappa_1}{C} (r_1^{q} + 1)^2 r_1^{-2p - 2q} \cdot r_1^{-g} 
		\ = \ K_{p, q} \cdot r_1^{-g},
	\end{equation}
	for some constant $K_{p, q} > 0$. Since $r_1 > 1$, this shows that the distribution decays geometrically in $r_1$.
	\hfill $\square$

	\section{Proof of Theorem \ref{gap_small}}\label{smallgap_section}
	
	When discussing small $g$, the left side of $H_{i + g}$ (included) and the right side of $H_{i}$ (included) are no longer independent. For example, if the legal decomposition on the left hand side of the gap $g$ is $H_n + H_{i + g + p} + H_{i + g}$ and the legal decomposition on the right side of the gap $g$ is $H_i + H_{i - p}  + H_{i + g - q} + H_1$, then once the two sides are combined together, the decomposition becomes illegal thus leading to an overcount since we have $H_{i + g + p} + H_{i + g} + H_{i + g - q} = H_{i + g + p + 1}$ by the recurrence. 
	
	However, for fixed $p, q$ and $g$, any illegal interaction caused by combining the left and right sides of the gap must occur inside a bounded window around the forced pair $H_{i+g}, H_i$. The size of this window depends only on $p, q$ and $g$, but not $n$, and mainly there are only finitely many types of left-hand-side and right-hand-side interruptions (two common examples of which are shown in Figure \ref{fig:ErrorTypes}). This implies that there are only finitely many types of local interactions.
	If we can prove in general, for each legal choice with certain fixed restrictions, the number of such legal decompositions has asymptotic form $C_{p,q,g} r_1^n + O(\rho^n)$ and after summing over $i$, the total contribution has form $C_{p,q,g} n r_1^n + O(n \rho^n)$, where $C_{p,q,g}$ is some non-zero constant in terms of $p, q,$ and $g$, then we can prove that the limiting probability density for small $g$ approaches a constant, since our total number of gaps $(H_{n - p} + H_{n - p - q}) (Cn + d - 1 + o(\gamma_1^n))$ = $D_{p,q} n r_1^{n} + o(n r_1^n)$ for some constant $D_{p,q} > 0$.
	
	\begin{figure}[h]
		\begin{center}
			\scalebox{.18}{\includegraphics{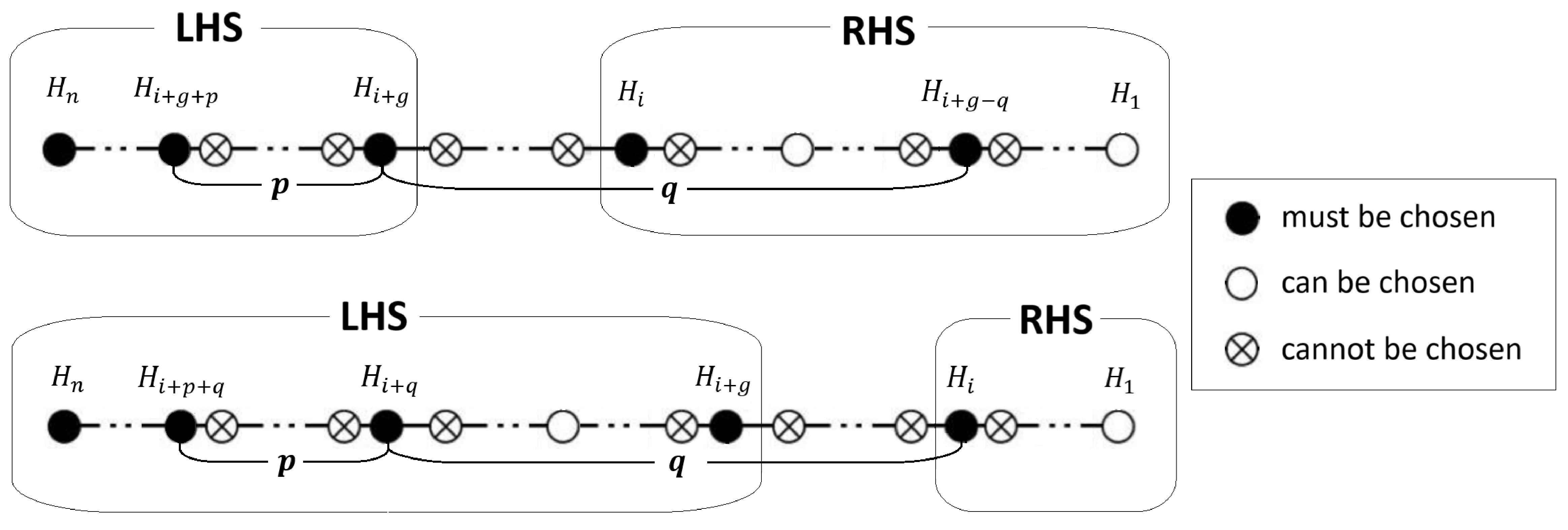}}
			\caption{\label{fig:ErrorTypes} Examples: two types of left and right hand side interactions.}
	\end{center}\end{figure}
	
	It is easy to prove certain types of restricted legal decompositions have such asymptotic form. For example, define $A_n$ as the total number of legal decompositions from $\set{H_1, \dots, H_n}$ where $H_n$ must be chosen and all other summands are free to be chosen, and define $B_n$ as the total number of legal decompositions from $\set{H_1, \dots, H_n}$ where $H_n$ and $H_1$ must be chosen and all other summands are free to be chosen. We can use the counting method from the proof of Lemma \ref{gap_no.} and get $A_n = H_{n + 1} - H_n$, and $B_n = H_{n + 1} - 2H_n + H_{n - 1}$. As $n \to \infty$, by Binet's formula $H_n = \kappa_{1} r_1^n + \kappa_{2} r_2^n + \dots + \kappa_{p + q + 1} r_{p + q + 1}^n = \kappa_1 \cdot r_1^{n} + O(\rho^n)$, substituting $H_n$ into the values of $A_n$ and $B_n$, we get
	\begin{equation}
		A_n \ = \ \kappa_1 \cdot r_1^{n + 1} - \kappa_1 \cdot r_1^{n} + O(\rho^n) \ = \ \alpha \cdot r_1^{n} + O(\rho^n),
	\end{equation}
	and
	\begin{equation}
		B_n \ = \ \kappa_{1} \cdot r_1^{n + 1} - 2 \kappa_{1} \cdot r_1^{n} + \kappa_{1} \cdot r_1^{n - 1} + O(\rho^{n}) \ = \ \beta \cdot r_1^{n} + O(\rho^n),
	\end{equation}
	where $\alpha$, $\beta$, and $\rho$ are all positive constants, and $|\rho| < r_1$. These are examples of legal choices with fixed endpoint restrictions.
	
	What about other middle-point restrictions --- say more generalized restrictions? Suppose, for example, that we define $C_n$ as the total number of legal decompositions from $\set{H_1, \dots, H_n}$ where certain restrictions are imposed, such as one fixed term being required to be chosen and another fixed term being required not to be chosen. To deal with such restrictions, we introduce a \textbf{finite-state transfer matrix}. This matrix method shows that the number of legal decompositions satisfying any fixed finite list of local restrictions has an asymptotic expansion of the form $C_{p,q,g} r_1^n + O(\rho^n)$, with $|\rho| < r_1$ and positive $C_{p,q,g}$ depends only on the restriction but not on $n$.
	
	A string $c_nc_{n-1} \dots c_1$ represents the decomposition $c_n H_n + c_{n-1} H_{n-1} + \dots + c_1 H_{1}$,
	where $c_j = 1$ means $H_j$ is chosen and $c_j = 0$ means $H_j$ is not chosen.
	
	For example, when $p = 1$ and $q = 2$, the recurrence is $H_{n+1} = H_n + H_{n-1} + H_{n-3}$, and the corresponding string pattern $1101$ is forbidden because it can be replaced by the single larger summand by the recurrence. Also, blocks containing $111$ are forbidden because it exceeds the allowed coefficient pattern.
	
	A \textbf{state} is a binary string of length $p+q$ that can legally appear as the most recent $p+q$ digits of a decomposition under the recursive PLRS legality rule. Let $\ell$ denote the total number of states: $S_1, S_2, \dots, S_{\ell}$.
	
	The total number of $(p + q)$-digit long binary strings is $2^{p + q}$.
	Now we explicitly list all the forms of illegal states and remove them from the list of all ($p + q$)-digit long binary strings, giving us all the legal states.
	\begin{remark}
		This transfer-matrix method only works when $p < q + 2$, meaning that everything in this section from this point onward is discussed under this condition. In Section \ref{fail}, we clearly explain why this method cannot be applied when $p\geq q+2$.
	\end{remark}
	
	\begin{theorem}\label{illegal state}
		If $p=q$ or $p=q+1$, then the illegal binary states are those containing patterns
		$$
		1\underbrace{00\dots 0}_{a \ 0\text{'s}}1,
		\qquad 0 \ \leq \ a \ \leq \ p-2.
		$$
		If $p<q$, then the illegal binary states are those containing patterns of forms shown above, 
		as well as all patterns of forms
		$$
		1\underbrace{00\dots 0}_{p-1 \ 0\text{'s}}1
		\underbrace{00\dots 0}_{b \ 0\text{'s}}1,
		\qquad p-1 \ \le \ b \ \le \ q-2.
		$$
	\end{theorem}

	\begin{proof}
		First, we prove every state containing one of the listed patterns is illegal. If a string contains 
		$$1 \ \underbrace{00\dots0}_{a \ 0\text{'s}} \ 1,$$
		where $0 \le a \le p - 2$, then a gap of size less than $p$ occurs. However, by Appendix \ref{mingap}, no legal decomposition can contain a gap of size less than $p$ for all cases $p < q + 2$.
		Altogether there are $p - 1$ such types of illegal strings.
		
		For cases $p<q$, besides the previously described $p - 1$ types of illegal strings, we need to remove another $q-p$ types of strings containing
		$$1 \ \underbrace{00\dots0}_{p - 1 \text{ 0's}} \ 1 \ \underbrace{00\dots0}_{b \text{ 0's}} \ 1,$$
		where $p-1 \le b \le q - 2$. Since the recurrence is a gap of size $p$ followed by a gap of size $q$, then these strings must have a sum no less than the recurrence, thus making these strings illegal.
		Altogether there are $q-1$ types of illegal strings for cases $p<q$.
		
		It remains to prove that no other types of illegal states occur. Suppose that a
		binary state avoids all the forbidden patterns listed above. We show that it is
		legal.
		
		Consider any substring beginning with a digit $1$. The coefficient pattern
		corresponding to the recurrence $H_{n+1}=H_n+H_{n-p}+H_{n-p-q}$
		is
		$$
		1\underbrace{00\dots0}_{p-1\text{ 0's}}1
		\underbrace{00\dots0}_{q-1\text{ 0's}}1.
		$$
		Because the state avoids every pattern
		$$
		1\underbrace{00\dots0}_{a\text{ 0's}}1,
		\qquad 0 \ \leq a \ \leq p-2,
		$$
		there can be no second $1$ among the next $p-1$ positions. Thus, after an
		initial $1$, the state must begin either with
		$$
		1\underbrace{00\dots0}_{p-1\text{ 0's}}
		$$
		or with
		$$
		1\underbrace{00\dots0}_{p\text{ 0's}}.
		$$
		
		In the first case, if $p<q$, the state avoids every pattern
		$$
		1\underbrace{00\dots0}_{p-1\text{ zeros}}1
		\underbrace{00\dots0}_{b\text{ zeros}}1,
		\qquad 0 \ \leq b \ \leq q-2.
		$$
		Therefore, after the second $1$, no further $1$ can occur until at least $q-1$ zeros have appeared. Consequently, the substring must
		begin with
		$$
		1\underbrace{00\dots0}_{p-1\text{ zeros}}1
		\underbrace{00\dots0}_{q-1\text{ zeros}}.
		$$
		Since a state has length $p+q$, the final $1$ of the full recurrence pattern
		does not occur inside the state. Thus the state does not equal or exceed the
		recurrence coefficient pattern and is legal.
		
		If $p=q$ or $p=q+1$, any third $1$ occurring early enough to form gaps sized smaller than $p$ is already excluded by the first family.
		Therefore, in these cases as well, every state avoiding the listed patterns is legal.
		
		Now consider the second case. The substring is strictly smaller than the recurrence coefficient pattern at position $p+1$, since it has a $0$ where the coefficient pattern has a $1$. And the remaining string is a downgraded version of an initial $1$ followed by $p-1$ or $p$ zeros.

		Thus every illegal state contains one of the listed forbidden patterns, and
		the list is complete.
	\end{proof}
	
	\begin{definition} \label{vector}
		Define $v_k$ to be the \textbf{state-count vector} for legal strings of length $k$. Its $j$\textsuperscript{th} component $v_k(S_j)$ records the number of legal strings of length $k$ whose latest added $p+q$ digits form the state $S_j$. Thus
		$$
		v_k \ : = \ 
		\begin{pmatrix}
			v_k(S_1) \\
			v_k(S_2) \\
			\vdots\\
			v_k(S_{\ell})
		\end{pmatrix}.
		$$
	\end{definition}
	Our strings are read from largest index $n$ to smallest index $p+q$, and appending one additional lower-index digit updates the state vector by the transfer matrix. Thus the vector evolves one digit at a time from length $p+q$ to length $n$.
	
	Although we write $v_{p+q}$ to indicate that the initial state
	has length $p+q$, the vector is not determined by the sum $p+q$ alone.
	Its components are indexed by the legal states for the recurrence
	$H_{n+1}=H_n+H_{n-p}+H_{n-p-q}$, and these legal states depend on the ordered pair $(p,q)$. Thus, in general, the vectors corresponding to $(p,q)=(1,3)$, $(2,2)$, and $(3,1)$ are different, even though $p+q=4$ in each case. Likewise, the matrices $M$ also depend on $(p,q)$.
	
	\begin{definition}  \label{matrix}
		We define the \textbf{transfer matrix} $M$ by
		$$
		M_{ab} \ := \
		\begin{cases}
			1, & \text{if state } S_a \text{ can legally follow state } S_b, \\
			0, & \text{otherwise.}
		\end{cases}
		$$
		Thus $v_{k+1}=Mv_k.$
	\end{definition}
	
	\begin{definition} \label{M1M0}
		We define $(M_0)_{ab}$ and $(M_1)_{ab}$ by
		$$
		(M_0)_{ab} \ := \
		\begin{cases}
			1, & \text{if state } S_a \text{ can legally follow state } S_b \text{ and the digit being added is } $0,$\\
			0, & \text{otherwise,}
		\end{cases}
		$$
		and
		$$
		(M_1)_{ab} \ := \
		\begin{cases}
			1, & \text{if state } S_a \text{ can legally follow state } S_b \text{ and the digit being added is } $1,$\\
			0, & \text{otherwise.}
		\end{cases}
		$$
		Then $M_{ab} = (M_0)_{ab} + (M_1)_{ab}$.
	\end{definition}
	
	\begin{lemma} \label{transfer matrix}
		If a legal string contains finitely many fixed local restrictions on its digits, then the number of such legal strings of length $n$ satisfying these restrictions has the asymptotic form
		$$
		C_{p,q,g}r_1^n+O(\rho^n),
		$$
		where $C_{p,q,g}$ is a positive constant depending on the restrictions, $r_1$ is the dominant eigenvalue of the transfer matrix $M$ (i.e., the leading root of the characteristic polynomial of our recurrence), and $0<\rho<r_1$.
		
	\end{lemma}
	
	\begin{proof}
		If no restrictions are implied (i.e., any summand from $H_1$ to $H_n$ is free to choose), then the number of legal decompositions is simply $\mathbf{1}^T M^{n - p - q} v_{p + q}$.

		When finitely many fixed local restrictions are imposed, we can count such legal decompositions by replacing $M$ by a limited number of $M_0$ or $M_1$. 
		Thus the number of legal strings under those restrictions has the form
		$$\mathbf{1}^T A_1 A_2 \dots A_{n - p - q} v_{p + q},$$
		where each $A_k$ can be either $M$, $M_0$, or $M_1$. Since the number of restriction matrices is fixed and independent of $n$, the total number of unrestricted transition steps is $n + O(1)$.
		
		In order to carry out the remaining analysis, we use the Perron--Frobenius theorem.
		It originated with Perron’s work \cite{Perron} on positive matrices and was later extended by Frobenius \cite{Frobenius1912} to non-negative matrices. Since the original papers are written in German, we refer the reader to \cite{Meyer} for an English version. 
		
		In order to use the Perron--Frobenius theorem, we have to first make sure that $M$ is primitive.\footnote{A matrix is primitive if, after multiplying it by itself enough times, every entry becomes positive.}
		Let $L=p+q$. Each state is a binary string of length $L$. We now show that for any two states $S_a$ and $S_b$, there is a legal path of length $2L$ from $S_a$ to $S_b$.
		
		Start from $S_a$ and append $L$ zeros to the end of the string. This must be legal, since adding zeros cannot create the forbidden pattern involving three $1$'s. After these $L$ steps, the current state is the all-zero state.
		Next append the $L$ digits of $S_b$. Doing this gets us to the state $S_b$. This second part also must be legal, since before the digits of $S_b$ are appended, the previous $L$ digits are all zero, so no forbidden pattern can occur across the boundary between the zeros and $S_b$. Moreover, the forbidden pattern has length $L+1=p+q+1$, while the block $S_b$ has length only $L$, so the forbidden pattern cannot occur entirely inside $S_b$.
		
		Thus, for every pair of states $S_a$ and $S_b$, there is a legal path of length $2L$ from $S_a$ to $S_b$. Therefore every entry of $M^{2L}$ is positive, so $M$ is primitive.
		
		Now we can use the Perron--Frobenius theorem, and get
		\begin{equation}
			M^n \ = \ r_1^nP+O(\rho^n)
		\end{equation}
		for some $\rho<r_1$, where $r_1$ is the dominant eigenvalue\footnote{The eigenvalue $r_1$ is also the largest characteristic root of our recurrence.} and $P$ is the \textbf{Perron projection matrix}. 
		If $u$ and $w$ are the right and left Perron eigenvectors, respectively, then
		\begin{equation}
			Mu \ = \ r_1u, \quad \text{ and} \quad \ w^TM \ = \ r_1w^T,
		\end{equation}
		and
		\begin{equation}
			P \ = \ \frac{uw^T}{w^Tu}.
		\end{equation}
		
		Since the number of restrictive matrices do not depend on $n$, inserting them changes only the leading constant $C$.
		Hence, the total number of legal strings satisfying the fixed restrictions also has the form
		$$C_{p,q,g} r_1^n+O(\rho^n).$$
	\end{proof}

	To count the total number of small gaps of size $g$, we first count the total number of legal decompositions where $H_n, H_i,$ and $H_{i + g}$ are chosen, all terms between $H_i$ and $H_{i+g}$ are not chosen, and all other terms are free to choose. This involves only $2 + (g-1) = g + 1$ fixed restrictions.\footnote{The condition $H_n=1$ is already included in $v_{p+q}$.}
	Building the string from higher index to lower index (i.e., from $H_n$ down to $H_1$), the number of legal decompositions with these restrictions in place are
	\begin{equation} \label{total small g}
		\mathbf{1}^T M^{i - 1} M_1 M_0^{g - 1} M_1 M^{n - p - q - i - g} v_{p+q}.
	\end{equation}
	Note that here $v_{p+q}$ denotes the initial state-count vector for the first $p+q$ highest-indexed digits, with the top digit corresponding to $H_n$ forced to be $1$.
	
	Since the asymptotic estimation from Perron--Frobenius can only be used when the power of $M$ is large, we need to divide the range of $i$ into three regions --- two edge regions $1 \le i \le f(n)$ and $n - p - q - g - f(n) + 1 \le i \le n - g$, as well as a middle region $f(n) + 1 \le i \le n - p - q - g - f(n)$, where $f(n) = \log n$, and only apply Lemma \ref{transfer matrix} to the middle region. We let $M^{i - 1} = M^a$ and $M^{n-p-q-i-g} = M^b$ for simplicity and call them \textbf{free blocks} since no restrictions are imposed in those strings. We can guarantee that $a = i-1$ and $b = n-p-q-i-g$ are both greater than $f(n)$. Note that $f(n) = \log n \to \infty$ as $n \to \infty$, implying that the lengths of both the free blocks approach infinity when $i$ is in the middle region. Then we can apply Perron--Frobenius Theorem on the two free blocks, and get
	\begin{equation}
		M^a \ = \ r_1^aP + O(\rho^a)
	\end{equation}
	and
	\begin{equation}
		M^b \ = \ r_1^bP + O(\rho^b).
	\end{equation}
	Now we substitute $M^a$ and $M^b$ into our counting expression, and the main term is
	\begin{equation}
		r_1^a \cdot r_1^b \ = \ r_1^{a+b} \ = \ r_1^{-p-q-g-1} \cdot r_1^n \ = \ C_{p,q,g} \cdot r_1^n.
	\end{equation}
	The error comes from the cases where at least one term is from the error term of the free blocks, and has the form of
	\begin{equation}
		\rho^a r_1^b 
		\ = \ r_1^{a+b} \left( \frac{\rho}{r_1} \right)^a 
		\ \le \ r_1^{a+b} \left( \frac{\rho}{r_1} \right)^{f(n)}
	\end{equation}
	or
	\begin{equation}
		\rho^b r_1^a \ = \ r_1^{a+b} \left( \frac{\rho}{r_1} \right)^b 
		\ \le \ r_1^{a+b} \left( \frac{\rho}{r_1} \right)^{f(n)}
	\end{equation}
	or
	\begin{equation}
		\rho^a \rho^b \ = \ r_1^{a+b} \left( \frac{\rho}{r_1} \right)^{a+b} 
		\ \le \ r_1^{a+b} \left( \frac{\rho}{r_1} \right)^{f(n)}.
	\end{equation}
	Thus the number of valid decompositions with this gap at position \(i\) has the form
	\begin{equation}
		C_{p,q,g} r_1^n + O\left( r_1^n \left( \frac{\rho}{r_1} \right)^{f(n)} \right)   \ = \ C_{p,q,g} r_1^n + o(r_1^n),
	\end{equation}
	where the constant depends only on the local configuration but not on $n$. Summing over the $n - 2f(n) - O(1) = n-o(n)$ numbers of middle-range values of $i$, we get
	\begin{equation}
		( n-o(n)) (C_{p,q,g} r_1^n + o(r_1^n)) \ = \ C_{p,q,g} n r_1^n + o(n r_1^n).
	\end{equation}
	Since the total number of gaps is
	$$
	D_{p,q} n r_1^n + o(n r_1^n)
	$$
	we obtain
	\begin{equation}
		P_n(g) \ = \ \frac{C_{p,q,g}}{D_{p,q}} + o(1),
	\end{equation}
	and
	\begin{equation}
		P(g) \ = \ \lim_{n \to \infty} P_n(g) \ = \ \frac{C_{p,q,g}}{D_{p,q}} \ = \ K_{p,q,g}.
	\end{equation}
	Thus the expression above gives the middle-region contribution to the probability.
	
	We now prove the insignificance of $i$ in either edge region. Since $p$, $g$, and $q$ are all constants, the total number of edge values of $i$ is bounded by $O(f(n))$.
	For such edge values of $i$, the number of legal decompositions satisfying the gap restrictions is at most the total number of legal decompositions using terms from $\set{H_1, \dots, H_n}$ with $H_n$ chosen. By the uniqueness of legal decompositions, this number is 
	\begin{equation}
		H_{n + 1} - H_n \ = \ H_{n - p} + H_{n - p - q}.
	\end{equation}
	By Binet's Formula we have $H_n$ is asymptotically equal to $\kappa_{1} r_1^n + O(\rho^n)$, where $|\rho| < r_1$.
	From this we get
	\begin{equation}
		H_{n - p} + H_{n - p - q} \ = \ O(r_1^n).
	\end{equation}
	Since there are only $O(f(n))$ edge values that we need to discuss, the total, edge contribution is $O(f(n) \cdot r_1^n)$.
	The total number of gaps is asymptotically equal to $n r_1^{n}$, meaning that
	\begin{equation}
		P_{n,edge}(g) \ = \ \frac{O(f(n) \cdot r_1^n)}{n r_1^n} \ = \ O\left( \frac{\log n}{n} \right) \to 0.
	\end{equation}
	Thus the edge cases do not affect the limiting probability.
	
	We now compute the limiting constant $K_{p, q, g}$.
	Substituting $M_a$ and $M_b$ into \eqref{total small g} and summing over about $n$ valid middle positions, we get
	\begin{equation}
		\begin{split}
			&n\mathbf{1}^T (r_1^aP) M_1 M_0^{g - 1} M_1 (r_1^bP) v_{p + q} + o(nr_1^n) \\
			&\quad = \ n r_1^{n -p - q - g - 1} \mathbf{1}^T P M_1 M_0^{g - 1} M_1 P v_{p + q}+ o(nr_1^n).
		\end{split}
	\end{equation}

	Then dividing by\footnote{We can also express the denominator as $\mathbf{1}^T M^{i - 1} M_1 M^{n - p - q - i} v_{p + q}$, which is asymptotically equal to $r_1^{n-p-q-1} \mathbf{1}^T P M_1 P v_{p + q}$. Then, we can get $P(g) = (r_1^{- g} \mathbf{1}^T P M_1 M_0^{g - 1} M_1 P v_{p + q}) / (\mathbf{1}^T P M_1 P v_{p + q})$. The result obtained is equivalent to that in Equation \eqref{P(g)small}.}
	$$C \kappa_{1} (r_1^{q} + 1) r_1^{-p - q} n r_1^n$$
	gives the limiting constant
	\begin{equation}\label{P(g)small}
		P(g) 
		\ = \ K_{p,q,g}
		\ = \ \frac{r_1^{- g - 1} \mathbf{1}^T P M_1 M_0^{g - 1} M_1 P v_{p + q}}{C \kappa_{1} (r_1^{q} + 1)}.
	\end{equation}
	
	\begin{remark}
		The formula in Theorem \ref{gap_small} contains an explicit factor of $r_1^{-g-1}$. We now determine when the remaining matrix factor is independent of gap size g.
	\end{remark}
	
	\begin{theorem}
		\label{thm:small_gap_decay}
		Assume that $p<q+2$.
		\begin{enumerate}
			\item If $p=q+1$, then
			$$\frac{P(g)}{P(g+1)} \ = \ r_1$$
			for every admissible gap $g \ge p$.
			\item If $p \le q$, then
			$$\frac{P(g)}{P(g+1)} \ = \ r_1$$
			for every gap such that $ p+1 \le g \le q-1$.
		\end{enumerate}
	\end{theorem}
	
	\begin{proof}
		When $p=q+1$, by Theorem \ref{mingap}, the minimum gap size is $p$. Therefore, a gap of size $q$ cannot occur in any legal decomposition, since $q < p$. The only possible interaction between the two sides of a forced gap would require the recurrence pattern, whose consecutive gaps are $p$ and $q$, to cross the forced gap. Since a gap of size $q$ is impossible, such an interaction cannot occur. Thus, the left- and right-hand sides are independent not only for $g>p$, but for every gap $g \ge p$. Consequently, the large-gap argument applies starting at $g = p$.
		
		Now let us consider the $p \le q$ case. If $p = q$ or $p = q -1$, then there is no integer $g$ satisfying $p+1 \le g \le q-1$. Thus, it remains only to consider the case $p \le q-2$.
		
		Define
		$$S(g) \ := \ \mathbf{1}^{T} P M_1 M_0^{g-1} M_1 P v_{p + q}.$$
		The quantity $S(g)$ represents the contribution of legal configurations containing two consecutive chosen summands separated by a gap of size $g$, which is of binary pattern
		$$1 0^{g - 1} 1.$$
		
		By Theorem \ref{gap_small}, we get
		$$P(g) \ = \ \frac{r_1^{-g-1}}{C \kappa_{1} (r_1^{q} + 1)} S(g),$$
		and hence
		$$\frac{P(g)}{P(g + 1)} \ = \ r_1 \frac{S(g)}{S(g + 1)}.$$
		Thus, it suffices to prove that $S(g) = S(g + 1)$ for all $p + 1 \le g \le q - 1$.
		
		\vspace{1em}
		
		\noindent
		\textbf{Case 1: The gap immediately preceding $H_{i + g}$ has size $p$.}
		
		In this case, the three consecutive summands have gap pattern
		$$1 0^{p - 1} 1 0^{g - 1} 1.$$ If $g \le q$, this is an illegal pattern described by Theorem \ref{illegal state}.
		Hence, no legal configuration with preceding gap $p$ contributes to $S(g)$ for $g \le q$. However, when $g = q + 1$, this forbidden pattern no longer occurs, meaning that this case may contribute to $S(q + 1)$. This explains why $S(q) \neq S(q+1)$ and the boundary ratio cannot be $r_1$.
		
		\vspace{1em}
		
		\noindent
		\textbf{Case 2: The gap immediately preceding $H_{i + g}$ has size larger than $p$}
		
		In this case, the local configuration has the form
		$$1 0^{d - 1} 1 0^{g - 1} 1, \qquad d \ > \ p.$$
		Inserting one additional zero between the two forced summands
		$$1 0^{d - 1} 1 0^{g - 1} 1 \ \longrightarrow \ 1 0^{d - 1} 1 0^{g} 1$$
		changes the forced gap from $g$ to $g + 1$, leaving the surrounding configuration unchanged and therefore preserving legality. Thus, the contribution of this case is independent of $g$, so we have
		$$S(g) \ = \ S(g + 1), \qquad p + 1 \ \le \ g \ \le \ q - 1.$$
		Equivalently,
		$$\frac{P(g)}{P(g + 1)} \ = \ r_1, \qquad p + 1 \ \le \ g \ \le \ q - 1.$$
		
		Note that when forced gap has size $p$, the configuration on one side of the gap may interact with the configuration on the other side through the recurrence pattern. This interaction disappears once the forced gap has size at least $p+1$. Therefore, $S(p) \neq S(p + 1)$.
		
		All in all, this proves that the non-boundary small-gap limiting probabilities decay geometrically with common ratio $r_1$ throughout the interval from $p + 1$ to $q$.
	\end{proof}
	
	Table \ref{tab:small_gap_decay} illustrates the geometric decay descibed in Theorem \ref{thm:small_gap_decay} for several choices of $p$ and $q$. In the ranges covered by the theorem, the numerical ratios $P(g)/P(g + 1)$ agree with the dominant root $r_1$, while the remaining values occur at the exceptional boundary gaps.
	
	\begin{table}[!ht]
		\centering
		\footnotesize
		\renewcommand{\arraystretch}{0.85}
		\setlength{\tabcolsep}{5pt}
		\begin{tabular}{cccccc}
			\toprule
			$p$ & $q$ & $g$ & $r_1$ & $P(g)$ & $P(g)/P(g+1)$ \\
			\midrule
			
			2 & 1 & 1 & 1.618033989 & 0           & 0           \\
			&   & 2 &             & 0.381966011 & 1.618033988 \\
			&   & 3 &             & 0.236067978 & 1.618033989 \\
			&   & 4 &             & 0.145898034 &             \\
			\midrule
			
			1 & 3 & 1 & 1.704902776 & 0.143768434 & 0.488058121 \\
			&   & 2 &             & 0.294572363 & 1.704902776 \\
			&   & 3 &             & 0.172779567 & 1.074601831 \\
			&   & 4 &             & 0.160784731 & 1.704902776 \\
			&   & 5 &             & 0.094307273 &             \\
			\midrule
			
			1 & 4 & 1 & 1.673648546 & 0.101562455 & 0.316645871 \\
			&   & 2 &             & 0.320744607 & 1.673648546 \\
			&   & 3 &             & 0.191643944 & 1.673648546 \\
			&   & 4 &             & 0.114506683 & 1.047669284 \\
			&   & 5 &             & 0.109296593 & 1.673648546 \\
			&   & 6 &             & 0.065304387 &             \\
			\midrule
			
			3 & 2 & 1 & 1.465571232 & 0           &             \\
			&   & 2 &             & 0           & 0           \\
			&   & 3 &             & 0.317672196 & 1.465571232 \\
			&   & 4 &             & 0.216756572 & 1.465571232 \\
			&   & 5 &             & 0.147899036 &             \\
			\midrule
			
			1 & 5 & 1 & 1.653634587 & 0.069613145 & 0.204600545 \\
			&   & 2 &             & 0.340239296 & 1.653634587 \\
			&   & 3 &             & 0.205752407 & 1.653634587 \\
			&   & 4 &             & 0.124424349 & 1.653634587 \\
			&   & 5 &             & 0.075242953 & 1.030476223 \\
			&   & 6 &             & 0.073017650 & 1.653634587 \\
			&   & 7 &             & 0.044155856 &             \\
			\midrule
			
			2 & 4 & 1 & 1.516784253 & 0           &             \\
			&   & 2 &             & 0.143271629 & 0.554522656 \\
			&   & 3 &             & 0.258369298 & 1.516784253 \\
			&   & 4 &             & 0.170340177 & 1.104677708 \\
			&   & 5 &             & 0.154198981 & 1.516784253 \\
			&   & 6 &             & 0.101661776 &             \\
			
			\bottomrule
		\end{tabular}
		\caption{Geometric decay among small-gap probabilities.}
		\label{tab:small_gap_decay}
	\end{table}
	
	\appendix
	
	\section{Minimum gap size for PLRS $H_{n + 1} = H_n + H_{n - p} + H_{n - p - q}$}\label{mingap}
	
	\begin{theorem}\label{gap_min}
		Let there be a recurrence $H_{n+1}=H_n+H_{n-p}+H_{n-p-q}$, with initial conditions to match a PLRS.
		The \textbf{minimum gap size} for $p \le q + 1$ is $p$. When $p \ge q + 2$, then a gap of size $p - 1$ is allowed.
	\end{theorem}

	\begin{lemma}\label{induction}
		When $p = q + 1$, we have $H_{k + 1} - H_k = H_{k - p + 1}$.
	\end{lemma}
	
	\begin{proof}
		This is proved by induction.
		
		\vspace{1em}
		
		\textbf{Base Case:} 
		Take $k = p$.
		Then the statement becomes  $
		H_{p + 1} - H_{p} \ = \ H_1 \ = \ 1,
		$
		which by the initial conditions is true. Thus the base case holds.
		
		\vspace{1em}
		
		\textbf{Inductive Step:} We assume that $H_{k + 1} - H_k = H_{k - p + 1}$ is true for all $k$, and we need to prove that it is also true for $k + p$ (i.e., $H_{k + p + 1} - H_{k + p} = H_{k + 1}$).
		
		Starting from the recurrence at index $k + p$, since $q = p - 1$, we get
		\begin{equation}
			H_{k + p + 1} \ = \ H_{k + p} + H_{k} + H_{k - p + 1}.
		\end{equation}
		After simplification, we have
		\begin{equation}
			H_{k + p + 1} - H_{k + p} \ = \ H_{k} + H_{k - p + 1}.
		\end{equation}
		Rearranging our assumption gives that,
		\begin{equation}
			H_{k + 1} \ = \ H_{k} + H_{k - p + 1}.
		\end{equation}
		Thus the statement holds for $k + p$, since $ H_{k + p + 1} - H_{k + p} = H_{k} + H_{k - p + 1} = H_{k + 1}$, completing the proof.
	\end{proof}
	
	\begin{proof}
		We now prove Theorem \ref{gap_min}.
		Suppose a legal decomposition contains \(H_k\), and suppose the next summand is \(H_j\), where \(j<k\). Since the total value of a decomposition with largest summand \(H_k\) must be less than \(H_{k+1}\), we must have $ H_k+H_j < H_{k+1} $, or equivalently, $ H_j < H_{k+1}-H_k. $
		Setting $ D_k := H_{k+1}-H_k $, thus we know that the smallest possible gap after \(H_k\) is determined by the largest index \(j<k\) for which $ H_j < D_k. $
		
		We now compute \(D_k\). For \(k \to \infty\), the recurrence gives $ H_{k+1} = H_k+H_{k-p}+H_{k-p-q}$, so
		\begin{equation}
			D_k \ = \ H_{k-p}+H_{k-p-q}.
		\end{equation}
		Since we want to figure out whether a gap of $p-1$ can occur, now the question is whether $H_j$ can equal $H_{k-(p-1)} = H_{k-p+1}$ or not. This is equivalent to finding whether \(H_{k-p+1}\) is less than, equal to, or greater than \(D_k\).
		If $H_{k-p+1} < D_k$, then a gap of size $p-1$ is allowed; otherwise, the minimum gap size is no less than $p$.
		To do this, we have to analyze three cases.
		
		\subsubsection{Case 1: $p \le q$}
		
		If $p \le q$, then $ k - p - q \le k - 2p. $
		So $H_{k - p - q}$ is not large enough to cause $D_k$ to be larger than $H_{k - p + 1}$. In fact,
		\begin{equation}
			D_k \ = \ H_{k - p} + H_{k - p - q} \ \le \ H_{k - p} + H_{k - 2p},
		\end{equation}
		but by the construction of the sequence,
		\begin{equation}
			H_{k - p + 1} \ = \ H_{k - p} + H_{k - 2p} + H_{k - 2p - q} \ > \ H_{k - p} + H_{k - 2p}
		\end{equation}
		for $k \to \infty$.
		
		Thus, $D_k < H_{k - p + 1}$ meaning that a gap of size $p - 1$ is impossible.
		
		Since $H_{k - p} < H_{k - p} + H_{k - p - q} = D_k$ is always true, we have $H_{k - p} < D_k$ for all cases, meaning that gap of size $p$ can be reached for all cases and the minimum gap size is $p$ in this case.
		
		\subsubsection{Case 2: $p = q + 1$}
		
		By definition, 
		\begin{equation}
			D_k \ = \ H_{k - p} + H_{k - p - q} \ = \ H_{k - p} + H_{k - 2p + 1}.
		\end{equation}
		By Lemma \ref{induction}, we have
		\begin{equation}
			D_k \ = \ H_{k + 1} - H_k\ = \ H_{k - p + 1}.
		\end{equation}
		Thus $H_{k - p + 1}$ is still not allowed, implying that a gap of size $p - 1$ is impossible and the minimum gap size is still $p$.
		
		\subsubsection{Case 3: $q \le p - 2$}
		
		Now $q$ is smaller, so $k - p - q$ is larger, then we have $ k - p - q \ge k - 2p + 2. $
		So $H_{k - p - q}$ is large enough to push $D_k$ past $H_{k - p + 1}$, meaning that $ D_k = H_{k - p} + H_{k - p - q} > H_{k - p + 1}. $ 
		Thus, a gap of size $p - 1$ is allowed after $H_k$.
	\end{proof}
	
	%%%%%%%%%%%%%%%%%%%%%%%%%%%%%%%%%%%%%%%%%%%%%%%%%%%%%%%%%%%%%%%%%%%%%%%%%%%%%%%%%%%%%%%%%%%%%%%%%%%%%%%%%%%%%%%%%%%%%%%%%%%%%%%%%%%%%%

	\section{Limitations of the Small-Gap Theorem for $p \ge q+2$} \label{fail}
	In this section we explain why Theorem \ref{gap_small} fails when $p \ge q + 2$. The main issue is that we cannot find an explicit rule that defines illegal states, thus we are incapable of listing legal states. An explicit legal state rule is essential, since our transfer matrix $M$ and initial vector $v$ are both formed using all legal states. Recall cases when $p<q+2$. We first list all the $(p+q)$-digit binary strings, construct rules for illegal strings in Theorem \ref{illegal state}, and remove all the $(p+q)$-digit binary strings containing illegal patterns, leaving all the legal states to form $M$ and $v$. The matrices formed in this way have two crucial features.
	\begin{itemize}
		\item The Perron eigenvalue $\lambda$ of the transfer matrix agrees with the dominant root $r_1$ of the recurrence (see data in Section \ref{intro}), meaning our transfer matrix perfectly describes the growth rate of the sequence.
		\item This matrix also records enough information to count gaps.
	\end{itemize}
	
	As for the $p \ge q + 2$ cases, \textbf{the illegal state rule is not uniform across the whole sequence}.
	\begin{theorem}
		For $p \ge q + 2$, a gap of size $p-1$ is forbidden among the initial-condition terms, but can appear later after the recurrence has begun.
	\end{theorem}
	The latter part of the theorem has been proved in Section \ref{mingap}. Now we only need to prove the former part is true. It suffices to prove that for all $ 1 \le k \le q+1$, we have $H_k + H_{k+p-1} = H_{k+p}$. This is obvious, since our initial conditions are
	\begin{align*}
		H_1       &\ = \ 1 \\
		H_2       &\ = \ 2 \\
		&\ \  \vdots \\
		H_{p+1}   &\ = \ p+1 \ = \ H_p + H_1
		\quad (k \ = \ 1)\\
		H_{p+2}   &\ = \ H_{p+1} + H_1 +1
		\ = \ H_{p+1} + H_2
		\quad (k \ = \ 2)\\
		H_{p+3}   &\ = \ H_{p+2} + H_2 +1
		\ = \ H_{p+2} + H_3
		\quad (k \ = \ 3)\\
		&\ \  \vdots \\
		H_{p+q+1} &\ = \ H_{p+q} + H_q +1
		\ = \ H_{p+q} + H_{q+1}
		\quad (k \ = \ q-1).
	\end{align*}
	Thus, strings containing patterns
	$$1\underbrace{00\dots 0}_{p-2 \ 0\text{'s}}1$$
	are difficult to be sorted as legal or illegal states, leading to the limitation of the method.

	%%%%%%%%%%%%%%%%%%%%%%%%%%%%%%%%%%%%%%%%%%%%%%%%%%%%%%%%%%%%%%%%%%%%%%%%%%%%%%%%%%%%%%%

	\section{Small Gap Counting Examples} \label{small gap example}
	
	\subsection{$(p,q,n)=(1,2,8)$}
	We first use the recurrence $ H_{n+1}=H_n+H_{n-1}+H_{n-3}$ as an example to check our small gap counting method. We know the initial conditions of this recurrence are
	$$
	H_1=1,\qquad H_2=2,\qquad H_3=4,\qquad H_4=7.
	$$
	Using the recurrence, we get the next few terms
	$$
	\begin{aligned}
		H_5 &\ =\ H_4+H_3+H_1
		&&\ =\ 7+4+1
		&&\ =\ 12,\\
		H_6 &\ =\ H_5+H_4+H_2
		&&\ =\ 12+7+2
		&&\ =\ 21,\\
		H_7 &\ =\ H_6+H_5+H_3
		&&\ =\ 21+12+4
		&&\ =\ 37,\\
		H_8 &\ =\ H_7+H_6+H_4
		&&\ =\ 37+21+7
		&&\ =\ 65,\\
		H_9 &\ =\ H_8+H_7+H_5
		&&\ =\ 65+37+12
		&&\ =\ 114.
	\end{aligned}
	$$
	We now list the legal, greedy decompositions and gap sizes for all integers in the range $$ [H_8,H_9) \ = \ [65,114)$$ in order to count the total number of gaps of size $1$ and $2$, which are the only small gap sizes for these values of $p$ and $q$. By brute force counting using Table \ref{brute force counting}, we know that altogether there are $33$ gaps of size $1$ and $34$ gaps of size $2$.
	\begin{longtable}{|r|l|l|}
		\caption{Greedy decompositions and gap sizes for integers in $[H_8,H_9)=[65,114)$ when $p=1$ and $q=2$.}
		\label{brute force counting} \\
		\hline
		$N$ & Greedy decomposition & Gap sizes \\
		\hline
		\endfirsthead
		\hline
		$N$ & Greedy decomposition & Gap sizes \\
		\hline
		\endhead
		65 & $H_8$ & -- \\
		66 & $H_8+H_1$ & $(7)$ \\
		67 & $H_8+H_2$ & $(6)$ \\
		68 & $H_8+H_2+H_1$ & $(6,1)$ \\
		69 & $H_8+H_3$ & $(5)$ \\
		70 & $H_8+H_3+H_1$ & $(5,2)$ \\
		71 & $H_8+H_3+H_2$ & $(5,1)$ \\
		72 & $H_8+H_4$ & $(4)$ \\
		73 & $H_8+H_4+H_1$ & $(4,3)$ \\
		74 & $H_8+H_4+H_2$ & $(4,2)$ \\
		75 & $H_8+H_4+H_2+H_1$ & $(4,2,1)$ \\
		76 & $H_8+H_4+H_3$ & $(4,1)$ \\
		77 & $H_8+H_5$ & $(3)$ \\
		78 & $H_8+H_5+H_1$ & $(3,4)$ \\
		79 & $H_8+H_5+H_2$ & $(3,3)$ \\
		80 & $H_8+H_5+H_2+H_1$ & $(3,3,1)$ \\
		81 & $H_8+H_5+H_3$ & $(3,2)$ \\
		82 & $H_8+H_5+H_3+H_1$ & $(3,2,2)$ \\
		83 & $H_8+H_5+H_3+H_2$ & $(3,2,1)$ \\
		84 & $H_8+H_5+H_4$ & $(3,1)$ \\
		85 & $H_8+H_5+H_4+H_1$ & $(3,1,3)$ \\
		86 & $H_8+H_6$ & $(2)$ \\
		87 & $H_8+H_6+H_1$ & $(2,5)$ \\
		88 & $H_8+H_6+H_2$ & $(2,4)$ \\
		89 & $H_8+H_6+H_2+H_1$ & $(2,4,1)$ \\
		90 & $H_8+H_6+H_3$ & $(2,3)$ \\
		91 & $H_8+H_6+H_3+H_1$ & $(2,3,2)$ \\
		92 & $H_8+H_6+H_3+H_2$ & $(2,3,1)$ \\
		93 & $H_8+H_6+H_4$ & $(2,2)$ \\
		94 & $H_8+H_6+H_4+H_1$ & $(2,2,3)$ \\
		95 & $H_8+H_6+H_4+H_2$ & $(2,2,2)$ \\
		96 & $H_8+H_6+H_4+H_2+H_1$ & $(2,2,2,1)$ \\
		97 & $H_8+H_6+H_4+H_3$ & $(2,2,1)$ \\
		98 & $H_8+H_6+H_5$ & $(2,1)$ \\
		99 & $H_8+H_6+H_5+H_1$ & $(2,1,4)$ \\
		100 & $H_8+H_6+H_5+H_2$ & $(2,1,3)$ \\
		101 & $H_8+H_6+H_5+H_2+H_1$ & $(2,1,3,1)$ \\
		102 & $H_8+H_7$ & $(1)$ \\
		103 & $H_8+H_7+H_1$ & $(1,6)$ \\
		104 & $H_8+H_7+H_2$ & $(1,5)$ \\
		105 & $H_8+H_7+H_2+H_1$ & $(1,5,1)$ \\
		106 & $H_8+H_7+H_3$ & $(1,4)$ \\
		107 & $H_8+H_7+H_3+H_1$ & $(1,4,2)$ \\
		108 & $H_8+H_7+H_3+H_2$ & $(1,4,1)$ \\
		109 & $H_8+H_7+H_4$ & $(1,3)$ \\
		110 & $H_8+H_7+H_4+H_1$ & $(1,3,3)$ \\
		111 & $H_8+H_7+H_4+H_2$ & $(1,3,2)$ \\
		112 & $H_8+H_7+H_4+H_2+H_1$ & $(1,3,2,1)$ \\
		113 & $H_8+H_7+H_4+H_3$ & $(1,3,1)$ \\
		\hline
	\end{longtable}

	Now we form the transition matrices $M, M_1$ and $M_0$.
	$$
	\begin{array}{c|ccccccc}
		M & S_1 & S_2 & S_3 & S_4 & S_5 & S_6 & S_7 \\
		\hline
		S_1=000 & 1 & 0 & 0 & 0 & 1 & 0 & 0 \\
		S_2=001 & 1 & 0 & 0 & 0 & 1 & 0 & 0 \\
		S_3=010 & 0 & 1 & 0 & 0 & 0 & 1 & 0 \\
		S_4=011 & 0 & 1 & 0 & 0 & 0 & 1 & 0 \\
		S_5=100 & 0 & 0 & 1 & 0 & 0 & 0 & 1 \\
		S_6=101 & 0 & 0 & 1 & 0 & 0 & 0 & 0 \\
		S_7=110 & 0 & 0 & 0 & 1 & 0 & 0 & 0
	\end{array}
	$$
	$$
	\begin{array}{c|ccccccc}
		M_1 & S_1 & S_2 & S_3 & S_4 & S_5 & S_6 & S_7 \\
		\hline
		S_1 & 0 & 0 & 0 & 0 & 0 & 0 & 0 \\
		S_2 & 1 & 0 & 0 & 0 & 1 & 0 & 0 \\
		S_3 & 0 & 0 & 0 & 0 & 0 & 0 & 0 \\
		S_4 & 0 & 1 & 0 & 0 & 0 & 1 & 0 \\
		S_5 & 0 & 0 & 0 & 0 & 0 & 0 & 0 \\
		S_6 & 0 & 0 & 1 & 0 & 0 & 0 & 0 \\
		S_7 & 0 & 0 & 0 & 0 & 0 & 0 & 0
	\end{array}
	\qquad
	\begin{array}{c|ccccccc}
		M_0 & S_1 & S_2 & S_3 & S_4 & S_5 & S_6 & S_7 \\
		\hline
		S_1 & 1 & 0 & 0 & 0 & 1 & 0 & 0 \\
		S_2 & 0 & 0 & 0 & 0 & 0 & 0 & 0 \\
		S_3 & 0 & 1 & 0 & 0 & 0 & 1 & 0 \\
		S_4 & 0 & 0 & 0 & 0 & 0 & 0 & 0 \\
		S_5 & 0 & 0 & 1 & 0 & 0 & 0 & 1 \\
		S_6 & 0 & 0 & 0 & 0 & 0 & 0 & 0 \\
		S_7 & 0 & 0 & 0 & 1 & 0 & 0 & 0
	\end{array}
	$$
	Here, matrix $M$ represents adding a new digit $0$ or $1$ to the end of our string. Matrix $M_0$ represents appending $0$, while $M_1$ represents appending $1$.
	
	Next, we construct the state-count vector $v_3$, which records the states determined by the three largest-indexed terms $H_8$, $H_7$, and $H_6$. Each component of $v_3$ counts the initial number of each state. Since $H_8$ is forced to be chosen and the pattern $111$ is illegal, the entries of $v_3$ are determined by the remaining legal states.
	$$
	v_3 \ = \ 
	\begin{pmatrix}
		0\\
		0\\
		0\\
		0\\
		1\\
		1\\
		1
	\end{pmatrix}
	\qquad
	\begin{array}{c|c}
		\text{State} & \text{Initial count} \\
		\hline
		S_1=000 & 0 \\
		S_2=001 & 0 \\
		S_3=010 & 0 \\
		S_4=011 & 0 \\
		S_5=100 & 1 \\
		S_6=101 & 1 \\
		S_7=110 & 1
	\end{array}
	$$

	Now we calculate the total number of gaps of size $1$ that go from $H_i$ to $H_{i + g}$. To do so, we first let $i$ run from $1$ to $4$ and then discuss $i = 5, 6, 7$ separately for computing the edge values.
	
	If $i = 1$, then the gap goes from $H_1$ to $H_{2}$ and we can calculate the total number of such gaps.
	$$
	\begin{array}{c|c|c|c|c|c|c}
		\text{State}
		& v_3
		& v_4=Mv_3
		& v_5=Mv_4
		& v_6=Mv_5
		& v_7=M_1v_6
		& v_8=M_1v_7 \\
		\hline
		S_1=000 & 0 & 1 & 2 & 3 & 0 & 0 \\
		S_2=001 & 0 & 1 & 2 & 3 & 5 & 0 \\
		S_3=010 & 0 & 1 & 1 & 3 & 0 & 0 \\
		S_4=011 & 0 & 1 & 1 & 3 & 4 & 8 \\
		S_5=100 & 1 & 1 & 1 & 2 & 0 & 0 \\
		S_6=101 & 1 & 0 & 1 & 1 & 3 & 0 \\
		S_7=110 & 1 & 0 & 1 & 1 & 0 & 0 \\
		\hline
		\mathbf{1}^T v_j
		& 3 & 5 & 9 & 16 & 12 & 8
	\end{array}
	$$
	Similarly if $i = 2, 3, 4$, then the total number of such gaps are shown in the tables below, respectively.
	$$
	\begin{array}{c|c|c|c|c|c|c}
		\text{State}
		& v_3
		& v_4=Mv_3
		& v_5=Mv_4
		& v_6=M_1v_5
		& v_7=M_1v_6
		& v_8=Mv_7 \\
		\hline
		S_1=000 & 0 & 1 & 2 & 0 & 0 & 0 \\
		S_2=001 & 0 & 1 & 2 & 3 & 0 & 0 \\
		S_3=010 & 0 & 1 & 1 & 0 & 0 & 0 \\
		S_4=011 & 0 & 1 & 1 & 3 & 4 & 0 \\
		S_5=100 & 1 & 1 & 1 & 0 & 0 & 0 \\
		S_6=101 & 1 & 0 & 1 & 1 & 0 & 0 \\
		S_7=110 & 1 & 0 & 1 & 0 & 0 & 4 \\
		\hline
		\mathbf{1}^T v_j
		& 3 & 5 & 9 & 7 & 4 & 4
	\end{array}
	$$
	
	$$
	\begin{array}{c|c|c|c|c|c|c}
		\text{State}
		& v_3
		& v_4=Mv_3
		& v_5=M_1v_4
		& v_6=M_1v_5
		& v_7=Mv_6
		& v_8=Mv_7 \\
		\hline
		S_1=000 & 0 & 1 & 0 & 0 & 0 & 0 \\
		S_2=001 & 0 & 1 & 2 & 0 & 0 & 0 \\
		S_3=010 & 0 & 1 & 0 & 0 & 0 & 0 \\
		S_4=011 & 0 & 1 & 1 & 3 & 0 & 0 \\
		S_5=100 & 1 & 1 & 0 & 0 & 0 & 3 \\
		S_6=101 & 1 & 0 & 1 & 0 & 0 & 0 \\
		S_7=110 & 1 & 0 & 0 & 0 & 3 & 0 \\
		\hline
		\mathbf{1}^T v_j
		& 3 & 5 & 4 & 3 & 3 & 3
	\end{array}
	$$
	
	$$
	\begin{array}{c|c|c|c|c|c|c}
		\text{State}
		& v_3
		& v_4=M_1v_3
		& v_5=M_1v_4
		& v_6=Mv_5
		& v_7=Mv_6
		& v_8=Mv_7 \\
		\hline
		S_1=000 & 0 & 0 & 0 & 0 & 0 & 1 \\
		S_2=001 & 0 & 1 & 0 & 0 & 0 & 1 \\
		S_3=010 & 0 & 0 & 0 & 0 & 0 & 0 \\
		S_4=011 & 0 & 1 & 1 & 0 & 0 & 0 \\
		S_5=100 & 1 & 0 & 0 & 0 & 1 & 0 \\
		S_6=101 & 1 & 0 & 0 & 0 & 0 & 0 \\
		S_7=110 & 1 & 0 & 0 & 1 & 0 & 0 \\
		\hline
		\mathbf{1}^T v_j
		& 3 & 2 & 1 & 1 & 1 & 2
	\end{array}
	$$
	
	Now we work on the case where $i = 5$. For the gap $H_6 \to H_5$, the initial state must be $H_8H_7H_6 = 101$, so we need a new vector $v'$ for edge calculation.
	$$
	v_3' \ = \ 
	\begin{pmatrix}
		0\\
		0\\
		0\\
		0\\
		0\\
		1\\
		0
	\end{pmatrix}
	\qquad
	\begin{array}{c|c}
		\text{State} & \text{Initial count} \\
		\hline
		S_1=000 & 0 \\
		S_2=001 & 0 \\
		S_3=010 & 0 \\
		S_4=011 & 0 \\
		S_5=100 & 0 \\
		S_6=101 & 1 \\
		S_7=110 & 0
	\end{array}
	$$
	$$
	\begin{array}{c|c|c|c|c|c|c}
		\text{State}
		& v_3'
		& v_4=M_1v_3'
		& v_5=Mv_4
		& v_6=Mv_5
		& v_7=Mv_6
		& v_8=Mv_7 \\
		\hline
		S_1=000 & 0 & 0 & 0 & 0 & 1 & 1 \\
		S_2=001 & 0 & 0 & 0 & 0 & 1 & 1 \\
		S_3=010 & 0 & 0 & 0 & 0 & 0 & 1 \\
		S_4=011 & 0 & 1 & 0 & 0 & 0 & 1 \\
		S_5=100 & 0 & 0 & 0 & 1 & 0 & 0 \\
		S_6=101 & 1 & 0 & 0 & 0 & 0 & 0 \\
		S_7=110 & 0 & 0 & 1 & 0 & 0 & 0 \\
		\hline
		\mathbf{1}^T v_j
		& 1 & 1 & 1 & 1 & 2 & 4
	\end{array}
	$$
	
	Note that $i=6$ implies the forbidden pattern $111$, so there are none.
	
	When $i = 7$ (i.e., for the gap $H_8 \to H_7$ of size $1$), the initial state must be $H_8H_7H_6 = 110$, so we construct another $v''$ just for the edge counting.
	$$
	v_3'' \ = \ 
	\begin{pmatrix}
		0\\
		0\\
		0\\
		0\\
		0\\
		0\\
		1
	\end{pmatrix}
	\qquad
	\begin{array}{c|c}
		\text{State} & \text{Initial count} \\
		\hline
		S_1=000 & 0 \\
		S_2=001 & 0 \\
		S_3=010 & 0 \\
		S_4=011 & 0 \\
		S_5=100 & 0 \\
		S_6=101 & 0 \\
		S_7=110 & 1
	\end{array}
	$$
	$$
	\begin{array}{c|c|c|c|c|c|c}
		\text{State}
		& v_3''
		& v_4=Mv_3''
		& v_5=Mv_4
		& v_6=Mv_5
		& v_7=Mv_6
		& v_8=Mv_7 \\
		\hline
		S_1=000 & 0 & 0 & 1 & 1 & 1 & 2 \\
		S_2=001 & 0 & 0 & 1 & 1 & 1 & 2 \\
		S_3=010 & 0 & 0 & 0 & 1 & 1 & 2 \\
		S_4=011 & 0 & 0 & 0 & 1 & 1 & 2 \\
		S_5=100 & 0 & 1 & 0 & 0 & 1 & 2 \\
		S_6=101 & 0 & 0 & 0 & 0 & 1 & 1 \\
		S_7=110 & 1 & 0 & 0 & 0 & 1 & 1 \\
		\hline
		\mathbf{1}^T v_j & 1 & 1 & 2 & 4 & 7 & 12
	\end{array}
	$$
	Summing over $i$, the matrix products for gaps of size $1$ sum to $33$, agreeing with the brute force count.
	$$
	\begin{array}{c|c|c|c}
		i & \text{Forced gap} & \text{Matrix product} & \text{Count} \\
		\hline
		1 & H_2\to H_1 & \mathbf{1}^T M_1M_1M^3v_3 & 8 \\
		2 & H_3\to H_2 & \mathbf{1}^T MM_1M_1M^2v_3 & 4 \\
		3 & H_4\to H_3 & \mathbf{1}^T M^2M_1M_1Mv_3 & 3 \\
		4 & H_5\to H_4 & \mathbf{1}^T M^3M_1M_1v_3 & 2 \\
		5 & H_6\to H_5 & \mathbf{1}^T M^4M_1v_3' & 4 \\
		6 & H_7\to H_6 & \text{impossible, since it would force }111 & 0 \\
		7 & H_8\to H_7 & \mathbf{1}^T M^5v_3'' & 12 \\
		\hline
		\multicolumn{3}{c|}{\text{Total number of gaps of size }1} & 33
	\end{array}
	$$
	
	Now we calculate the total number of gaps of size $2$. The method is straightforward and similar to the one used for gaps of size $1$.
	For $g=2$, we force the pattern
	$$
	H_{i+2} \ = \ 1,\qquad H_{i+1} \ = \ 0,\qquad H_i \ = \ 1.
	$$
	Thus the forced block is $M_1M_0M_1$. Summing over $i$, we get $34$ gaps of size $2$, matching the brute force count.\footnote{For $g=2$ and $i=5$, the forced gap is $H_7 \to H_5$, so the forced pattern is $H_7H_6H_5=101$. Since $H_8$ is already forced, this creates $H_8H_7H_6H_5=1101$, which is exactly the recurrence block and thus is illegal.}
	$$
	\begin{array}{c|c|c|c}
		i & \text{Forced gap} & \text{Matrix product} & \text{Count} \\
		\hline
		1 & H_3\to H_1 & \mathbf{1}^T M_1M_0M_1M^2v_3 & 4 \\
		2 & H_4\to H_2 & \mathbf{1}^T MM_1M_0M_1Mv_3 & 6 \\
		3 & H_5\to H_3 & \mathbf{1}^T M^2M_1M_0M_1v_3 & 3 \\
		4 & H_6\to H_4 & \mathbf{1}^T M^3M_1M_0v_3' & 5 \\
		5 & H_7\to H_5 & \mathbf{1}^T M^4M_1v_3'' & 0 \\
		6 & H_8\to H_6 & \mathbf{1}^T M^5v_3' & 16 \\
		\hline
		\multicolumn{3}{c|}{\text{Total number of gaps of size }2} & 34
	\end{array}
	$$
	
	\subsection{$(p,q,n)=(2,1,9)$}
	
	For $(p,q)=(2,1)$, the recurrence is
	$$
	H_{n+1} \ = \ H_n+H_{n-2}+H_{n-3}.
	$$
	The initial conditions are
	$$
	H_1 \ = \ 1,\qquad H_2 \ = \ 2,\qquad H_3 \ = \ 3,\qquad H_4 \ = \ 5.
	$$
	Applying the recurrence gives
	$$
	H_5 \ = \ 8,\quad H_6 \ = \ 13,\quad H_7 \ = \ 21,\quad
	H_8 \ = \ 34,\quad H_9 \ = \ 55,\quad H_{10} \ = \ 89.
	$$
	Thus, we consider the legal decompositions of the integers in
	$$
	[H_9,H_{10}) \ = \ [55,89),
	$$
	and counting the gaps in these decompositions gives
	$$
	\begin{array}{c c c c}
		\hline
		p & q & g & P_9(g) \\
		\hline
		2 & 1 & 1 & 0/71 \\
		2 & 1 & 2 & 33/71 \\
		\hline
		2 & 1 & 3 & 18/71 \\
		2 & 1 & 4 & 10/71 \\
		2 & 1 & 5 & 5/71 \\
		2 & 1 & 6 & 3/71 \\
		2 & 1 & 7 & 1/71 \\
		2 & 1 & 8 & 1/71 \\
		\hline
	\end{array}
	$$
	
	The recurrence block is $1011$ and the pattern $11$ is forbidden. Thus, the legal states of length $p+q=3$ are
	$$
	S_1=000,\qquad
	S_2=001,\qquad
	S_3=010,\qquad
	S_4=100,\qquad
	S_5=101.
	$$
	
	Now we form the transfer matrix $M$.
	$$
	\begin{array}{c|ccccc}
		M & S_1 & S_2 & S_3 & S_4 & S_5 \\
		\hline
		S_1=000 & 1 & 0 & 0 & 1 & 0 \\
		S_2=001 & 1 & 0 & 0 & 1 & 0 \\
		S_3=010 & 0 & 1 & 0 & 0 & 1 \\
		S_4=100 & 0 & 0 & 1 & 0 & 0 \\
		S_5=101 & 0 & 0 & 1 & 0 & 0
	\end{array}
	$$
	Separating transitions according to whether the appended digit is
	$1$ or $0$ gives $M_1$ and $M_0$.
	$$
	M_1 \ = \ 
	\begin{pmatrix}
		0 & 0 & 0 & 0 & 0 \\
		1 & 0 & 0 & 1 & 0 \\
		0 & 0 & 0 & 0 & 0 \\
		0 & 0 & 0 & 0 & 0 \\
		0 & 0 & 1 & 0 & 0
	\end{pmatrix}
	\qquad
	M_0 \ = \ 
	\begin{pmatrix}
		1 & 0 & 0 & 1 & 0 \\
		0 & 0 & 0 & 0 & 0 \\
		0 & 1 & 0 & 0 & 1 \\
		0 & 0 & 1 & 0 & 0 \\
		0 & 0 & 0 & 0 & 0
	\end{pmatrix}
	$$
	
	We next construct the initial state-count vector $v_3$. Since $H_9$
	is forced to be chosen, the first digit of the state corresponding
	to $H_9,H_8,H_7$ must be $1$.
	$$
	v_3 \ = \ 
	\begin{pmatrix}
		0\\
		0\\
		0\\
		1\\
		1
	\end{pmatrix}
	\qquad
	\begin{array}{c|c}
		\text{State} & \text{Initial count} \\
		\hline
		S_1=000 & 0 \\
		S_2=001 & 0 \\
		S_3=010 & 0 \\
		S_4=100 & 1 \\
		S_5=101 & 1
	\end{array}.
	$$
	
	When counting the number of gaps of size $2$, for boundary cases in which 
	$H_9$ and $H_7$ are both forced to choose, we invent an special vector $v_3'$ only for edge counting.
	$$
	v_3' \ = \ 
	\begin{pmatrix}
		0\\
		0\\
		0\\
		0\\
		1
	\end{pmatrix}
	$$
	
	Gaps of size $1$ are impossible because they would produce the
	forbidden pattern $11$. We now count gaps of size $2$.
	
	$$
	\begin{array}{c|c|c|c}
		i & \text{Forced gap} & \text{Matrix product} & \text{Count} \\
		\hline
		1 & H_3 \to H_1
		& \mathbf{1}^{T}M_1M_0M_1M^3v_3
		& 5 \\
		
		2 & H_4 \to H_2
		& \mathbf{1}^{T}MM_1M_0M_1M^2v_3
		& 3 \\
		
		3 & H_5 \to H_3
		& \mathbf{1}^{T}M^2M_1M_0M_1Mv_3
		& 4 \\
		
		4 & H_6 \to H_4
		& \mathbf{1}^{T}M^3M_1M_0M_1v_3
		& 3 \\
		
		5 & H_7 \to H_5
		& \mathbf{1}^{T}M^4M_1M_0v_3'
		& 5 \\
		
		6 & H_8 \to H_6
		& \text{impossible, since it would force the pattern }11
		& 0 \\
		
		7 & H_9 \to H_7
		& \mathbf{1}^{T}M^6v_3'
		& 13 \\
		\hline
		\multicolumn{3}{c|}{\text{Total number of gaps of size }2}
		& 33
	\end{array}
	$$
	We can see that the result give by matrix method agrees with that from brute force counting.

	%%%%%%%%%%%%%%%%%%%%%%%%%%%%%%%%%%%%%%%%%%%%%%%%%%%%%%%%%%%%%%%%%%%%%%%%%%%%%%%%%%%%%%%%%%%%%%%%%%%%%%%%%%%%%%%%%%%%%%%%%%%%%%%%%%%%%%%%%%%%%%%%%%%%%%%%%%%%%%
	
\section{The distribution of the average gap measure for PLRS $H_{n + 1} = H_n + H_{n - p}$} \label {Kangaroo}
	
	In the previous section, we studied the sequence $H_{n + 1} = H_n + H_{n - 2}$. In this section, we generalize the finding and study the distribution of the average gap measure for the recurrence in the form of $H_{n + 1} = H_n + H_{n - p}$ with initial conditions to match those for a PLRS sequence. We consider integers $N$ in the interval $[H_n, H_{n + 1})$ and analyze the behavior of the gap distribution in their legal decompositions as $n \to \infty$. Then, we show that the distribution decays geometrically.
	
	\subsection{The minimum gap size}
	According to the definition of a legal decomposition and the structure of the recurrence, a decomposition of $N$ is legal if and only if the gap between consecutive summands has size at least $p+1$.

	\subsection{Total number of gaps}
	
	\begin{corollary}
		The total number of gaps in the legal decompositions of all $N \in [H_n, H_{n + 1})$ is
		$$H_{n - p} (Cn + d - 1 + o(\gamma_1^n)).$$ 
	\end{corollary}
	
	\begin{proof}
		By Theorem \ref{generallek}, the average number of summands in the legal decompositions of all $N \in [H_n, H_{n + 1})$ is $Cn + d + o(\gamma_1^n)$. Since a decomposition with $k$ summands contains $k - 1$ gaps, the average number of gaps is $Cn + d - 1 + o(\gamma_1^n).$
		
		The total number of gaps is equal to the average number of gaps multiplied by the total number of decompositions. By the uniqueness of legal decompostiions, the total number of decompositions in $[H_n, H_{n + 1})$ is equal to the number of integers in this interval, namely $H_{n + 1} - H_n = H_{n - p}$.
	\end{proof}
	
	\subsection{Number of occurrences of a gap of size $g$}
	
	\begin{theorem}
		The number of occurrences of a gap of size $g$ (where $g \ge p+1$) in the legal decompositions of all $N \in [H_n, H_{n + 1})$ is
		$$\sum_{i = 1}^{n - g} H_{i - p} \cdot H_{n - i - g - 2p}.$$
	\end{theorem}
	
	\begin{proof}
		We prove this theorem by counting the number of possible arrangements to the left and to the right of a gap of size $g$. Let $H_i$ and $H_{i + g}$ be consecutive summands in the decomposition.
		
		First, the number of arrangements to the right of the gap (including $H_i$) equals the number of legal choices from $\set{H_1, H_2, \dots, H_i}$ with $H_i$ chosen. By the existence and uniqueness of legal decompositions for this PLRS, these choices are in bijection with the total number of integers in the interval $[H_i, H_{i + 1})$. Thus, the number of such arrangement is $H_{i + 1} - H_i = H_{i - p}$.
		
		Next, the number of arrangements to the left of the gap (including $H_{i + g}$) correspond to the number of legal choices from $\set{H_{i + g}, H_{i + g + 1}, \dots, H_n}$ with $H_{i + g}$ and $H_n$ chosen. By shifting indices, this bijects with the number of legal choices from $\set{H_1, H_2, \dots, H_{n - i - g + 1}}$ with $H_1$ and $H_{n - i - g + 1}$ must chosen. In order to compute this quantity, we find the number of arrangements where $H_{n - i - g + 1}$ must be used and subtract the number of arrangements where $H_{n - i - g + 1}$ must be used but $H_1$ cannot be used.
		
		The number of arrangements where $H_{n - i - g + 1}$ must be used bijects with the number of integers in $[H_{n-i-g+1}, H_{n-i-g+2})$, which is $H_{n - i - g + 2} - H_{n - i - g + 1} = H_{n - i - g - p + 1}$. And the number of arrangements where $H_{n - i - g + 1}$ must be used but $H_1$ cannot be used bijects with the number of legal choices from $ \set {H_2, H_3, ... H_{n-i-g+1}}$  where $H_{n-i-g+1}$ must be chosen, and thus forms a bijection with the number of legal choices from $ \set {H_1, H_2, ..., H_{n - i - g}}$ where $H_{n - i - g}$ must be chosen by shifting indices. This is then equal to the number of integers in $[H_{n-i-g}, H_{n-i-g+1})$, which is $H_{n - i - g + 1} - H_{n - i - g} = H_{n - i - g - p}$.
		Thus, the valid left-hand arrangements is $H_{n - i - g - p + 1} - H_{n - i - g - p} = H_{n - i - g - 2p}$.
		
		Multiplying the number of choices on the left and right and summing over all possible positions $i$ of the gaps gives the total number of occurrences of a gap of size $g$ is 
		$$\sum_{i = 1}^{n - g} H_{i - p} \cdot H_{n - i - g - 2p}.$$
		
		In cases where negative indices appear, we extend the recurrence backward until $H_{-2p}$.
	\end{proof}
	
	\subsection{The distribution of the average gap measure and its geometric decay}
	
	\begin{theorem}
		The probability of choosing a gap of size $g$ from decompositions of all $N \in [H_n, H_{n + 1})$ is
		\begin{equation}\label{asymptoticgeneral}
			P_n(g) \ = \ \frac{\sum_{i = 1}^{n - g} H_{i - p} \cdot H_{n - i - g - 2p}}{H_{n - p} (Cn + d - 1 + o(\gamma_1^n))},
		\end{equation}
		and the distribution of $P_n(g)$ geometrically decays as $n \to \infty$.
	\end{theorem}

	\begin{proof}
		By the generalized Binet formula,
		$$
		H_n \ = \ \kappa_{r_1,1} r_1^n + \kappa_{r_2,1} r_2^n + \dots + \kappa_{r_{p + 1},1} r_{p + 1}^n,
		$$
		where $r_1$ is the unique dominant root and $|r_2|, |r_3|, \dots, |r_{p + 1}| < r_1$.
		We can also prove that the leading coefficient $\kappa_{r_1,1} > 0$. Since we have proved that $r_1$ is the unique dominant root, and the generalized Binet formula gives
		$$
		H_n \ = \ \kappa_{r_1,1}r_1^n+o(r_1^n),
		$$
		by dividing $r_1^n$ we obtain
		$$
		\frac{H_n}{r_1^n} \ = \ \kappa_{r_1,1}+o(1).
		$$
		The left-hand side is positive for all $n$, since both $H_n>0$ and $r_1$ are positive, thus forcing the limiting leading coefficient must be positive. (It cannot be zero, or else $H_n = o(r_1^n)$, contradicting the fact that $r_1$ is the dominant growth rate of the sequence.)
		
		Each product
		$H_{i-p} H_{n-i-g-2p}$ in the summation in \eqref{asymptoticgeneral} expands into $(p + 1)^2$ terms. Let $ a = i - p$ and $ b = n - i - g - 2p$ for simplicity.
		By the closed-form of $H_n$, we obtain
		\begin{equation} \label{p_expansion}
			\begin{aligned}
				H_a H_b
				& \ = \ (\kappa_{r_1,1} r_1^{a} + \dots + \kappa_{r_{p + 1},1} r_{p + 1}^{a})(\kappa_{r_1,1} r_1^{b} + \dots + \kappa_{r_{p + 1},1} r_{p + 1}^{b}) \\
				& \ = \ \kappa_1^2 r_1^{a} r_1^{b}
				+ \kappa_1 \kappa_2 r_1^{a} r_2^{b}
				+ \dots 
				+ \kappa_{p + 1}^2 r_{p + 1}^{a} r_{p + 1}^{b} \\
				& \ = \ r_1^{a} r_1^{b} \cdot \Bigl( \kappa_1^2 
				+ \kappa_1 \kappa_2 \left( \frac{r_2}{r_1} \right) ^{b}
				+\dots
				+ \kappa_{p + 1}^2 \left( \frac{r_{p + 1}}{r_1} \right) ^{a} \left( \frac{r_{p + 1}}{r_1} \right) ^{b} \Bigr).
			\end{aligned}
		\end{equation}

		The significance of the non-leading pieces are not universally the same in the summation in \eqref{p_expansion}. We know that the magnitudes of $r_2/r_1, r_3/r_1, \dots, r_{p + 1}/r_1$ are all between $0$ and $1$, so when the index is close to either the smallest or largest index in the summation (i.e., when $a$ or $b$ is small), the non-leading contribution may be of the same order as the leading contribution up to a constant factor, while in the middle range, the contribution of the error terms can be insignificant since $a$ and $b$ are both relatively large. Thus, to do the rigorous analysis, we need to separate the summation into three regions:
		$$
		\sum_{i=1}^{n-g} H_a H_b \ = \ \sum_{i=1}^{f(n)} H_a H_b + \sum_{i=f(n)+1}^{n-g-f(n)-1} H_a H_b + \sum_{i=n-g-f(n)}^{n-g} H_a H_b,
		$$
		where we let $f(n) = \log n$ such that, when $n \to \infty$, we have $f(n) \to \infty$. Let $h(n) = n$ and by L'H\^{o}pital's Rule, we get
		$$
		\lim_{n \to \infty} \frac{f(n)}{h(n)} 
		\ = \ \lim_{n \to \infty} \frac{f'(n)}{h'(n)} 
		\ = \ \lim_{n \to \infty} \frac{1/n}{1}
		\ = \ 0
		$$
		(i.e., $f(n) = o(n)$).
		
		The first and third sums correspond to the left and right edge regions respectively, each of which contains at most $f(n)$ terms.
		Since each non-leading product has order at most $r_1^{n-g-3p}$ (i.e., $r_1^{a + b}$), we get the total left edge size as
		$$
		\sum_{i=1}^{f(n)} H_{i-p} H_{n-i-g-2p}
		\ = \ O(f(n) \cdot r_1^{n - g - 3p}).
		$$
		
		As for the denominator of \eqref{asymptoticgeneral},
		$$H_{n-p} (Cn + d - 1 + o(\gamma_1^n)) \ = \ \Theta \left( n \cdot r_1^{n - p}\right). $$
		By forming the ratio, we obtain that 
		$$\frac{\text{left edge contribution}}{\text{denominator}}
		\ = \ O \left( \frac{f(n) \cdot r_1^{n - g - 3p}}{n \cdot r_1^{n-p}} \right)
		\ = \ O \left( \frac{f(n)}{n} \cdot r_1^{- g - 2p}\right). $$
		As $f(n)/n \to 0$, the left edge contribution is negligible. Similarly, we find that the right edge contribution is also negligible.
		
		We now consider the middle range
		$$ f(n) +1 \ \leq \ i \ \leq \ n-g-f(n) - 1.$$
		In this region, $i - p \geq f(n) - p + 1$, or $i - p \geq f(n) + O(1)$, and $n-i-g-2p \geq f(n) - 2p + 1$, or $n-i-g-2p \geq f(n) + O(1).$

		Let $\lambda$ largest out of $|r_2|, |r_3|, \dots, |r_{p + 1}| < r_1$. 		
		The error term arises from all other $(p + 1)^2 - 1$ terms other than the main term, each of which contains at least one of $(r_2/r_1)^k, (r_3/r_1)^k, \dots, (r_{p + 1}/r_1)^k$, where $k$ is either equal to $a$ (i.e., $i - p$) or $b$ (i.e., $n - i - g - 2p$). Thus, every non-dominant contribution in \eqref{p_expansion} gains a factor of $(\lambda/r_1)^{i-p}$ or $(\lambda/r_1)^{n-i-g-2p}$, both of which are bounded above by $(\lambda/r_1)^{f(n)+O(1)} = C(\lambda/r_1)^{f(n)}$ for some constant $C$. Since constants are absorbed into the Big-O term, we then get
		$$
		H_{i-p} H_{n-i-g-2p}
		\ = \ \kappa_1^2 r_1^{n-g-3p} + O\left(r_1^{n-g-3p} \left(\frac{\lambda}{r_1}\right)^{f(n)}\right).
		$$
		Because $2f(n) + O(1) = O(n)$, summing the $n-2f(n)+O(1)$ terms over the middle range gives
		$$
		\sum_{i=f(n)+1}^{n-g-f(n)-1} H_{i-p} H_{n-i-g-2p}
		\ = \ \kappa_1^2 r_1^{n-g-3p} ( n - 2f(n) + O(1) )
		+ O\left(n r_1^{n-g-3p} \left(\frac{\lambda}{r_1}\right)^{f(n)}\right).
		$$
		Since $f(n) = \log n$, we have $(\lambda/r_1)^{f(n)} \to 0$ as $n \to \infty$, thus giving
		$$
		\sum_{i=f(n)+1}^{n-g-f(n)-1} H_{i-p} H_{n-i-g-2p}
		\ = \ \kappa_1^2 r_1^{n-g-3p} (n-2f(n)+O(1)) + o(n r_1^{n-g-3p}).
		$$
		Because $2f(n) + O(1) = o(n)$, the equation above can be further simplified to
		\begin{align*}
			\sum_{i=f(n)+1}^{n-g-f(n)-1} H_{i-p} H_{n-i-g-2p}
			& \ = \ \kappa_1^2 r_1^{n-g-3p} (n+o(n)) + o(n r_1^{n-g-3p}) \\
			& \ = \ \kappa_1^2 r_1^{n-g-3p} n + \kappa_1^2 o(n r_1^{n-g-3p}) + o(n r_1^{n-g-3p}) \\
			& \ = \ \kappa_1^2 r_1^{n-g-3p} n + o\!\left(n r_1^{n-g-3p}\right).
		\end{align*}
		
		Substituting this into \eqref{asymptoticgeneral}, and using
		$ H_{n-p} = \kappa_1 r_1^{n-p} + o(r_1^{n-p}), $
		we obtain
		$$
		P_n(g)
		\ = \ \frac{\kappa_1^2 r_1^{n-g-3p} n + o(n r_1^{n-g-3p})}
		{(\kappa_1 r_1^{n-p} + o(r_1^{n-p})) (Cn + d - 1 + o(\gamma_1^n))}.
		$$
		The denimonator can be written as
		$$
		(\kappa_1 r_1^{n-p} + o(r_1^{n-p})) (Cn + O(1)) \ \sim \ (\kappa_1 r_1^{n-p} + o(r_1^{n-p}))Cn,
		$$
		so
		$$
		P_n(g)
		\ = \ \frac{\kappa_1^2 n r_1^{n-g-3p} + o(n r_1^{n-g-3p})}
		{\kappa_1 r_1^{n-p} Cn + o(nr_1^{n-p})}.
		$$
		Simplifying gives us
		$$P_n(g)
		\ = \ \frac{\kappa_1}{C} r_1^{-g-2p} + o(1).$$
		Taking the limit as $n \to \infty$, we obtain
		$$
		\lim_{n \to \infty} P_n(g) \ = \ K r_1^{-g},
		$$
		for some constant $K > 0$. Since $r_1 > 1$, this shows that the distribution
		decays geometrically in $r_1$.
	\end{proof}
	
	\subsection{Finding the constant $K$ in the geometric decay}
	
	From the asymptotic analysis, we have
	\begin{equation*}
	\lim_{n\to\infty} P_n(g)
	\ = \ \frac{\kappa_1}{C}\, r_1^{-g-2p}.
	\end{equation*}
	Since legal decompositions require gaps of size at least $p + 1$, the limiting probabilities must satisfy
	\begin{equation*}
	\sum_{g = p + 1}^{\infty} \lim_{n\to\infty} P_n(g) \ = \ 1.
	\end{equation*}
	Substituting the expression above yields
	\begin{equation*}
	\frac{\kappa_1}{C} \sum_{g=p + 1}^{\infty} r_1^{-g-2p} \ = \ 1.
	\end{equation*}
	Evaluating the geometric series gives
	\begin{equation*}
	\sum_{g=p + 1}^{\infty} r_1^{-g-2p} \ = \ \sum_{m = 3p + 1}^{\infty} r_1^{-m} \ = \ \frac{r_1^{-(3p + 1)}}{1 - r_1^{-1}},
	\end{equation*}
	where $m = g+2p$, and hence
	\begin{equation*}
	\frac{\kappa_1}{C} \cdot \frac{r_1^{-(3p + 1)}}{1 - r_1^{-1}} \ = \ 1,
	\end{equation*}
	which implies
	\begin{equation*}
	\frac{\kappa_1}{C} \ = \ r_1^{3p} (r_1 - 1).
	\end{equation*}
	Using the identity $r_1^{p + 1} - r_1^{p} - 1 = 0$, we obtain $r_1^{p}(r_1 - 1) = 1$, and thus $r_1^{3p}(r_1 - 1) = r_1^{2p}$. Therefore,
	\begin{equation*}
	\frac{\kappa_1}{C} \ = \ r_1^{2p},
	\end{equation*}
	and the limiting distribution simplifies to
	\begin{equation*}
	\lim_{n\to\infty} P_n(g) \ = \ r_1^{-g}, \quad g \ge p + 1,
	\end{equation*}
	which is a geometric distribution with ratio $1/r_1$.
	
	%%%%%%%%%%%%%%%%%%%%%%%%%%%%%%%%%%%%%%%%%%%%%%%%%%%%%%%%%%%%%%%%%%%%%%%%%%%%%%%%%%%%%%%%%%%%%%%%%%%%%%%%%%%%%%%%%%%%%%%%%%%%%%%%%%%%%%%%%%%%%%%%%%%%%%%%%%%%%%%%%%%%%%

	No funding was received.
	%%%%%%%%%%%%%%%%%%%%%%%%%%%%%%%%%%%%%%%%%%%%%%%%%%%%%%%%%%%%%%%%%%%%%%%%%%%%%%%%%%%%%%%%%%%%%%%%%%%%%%%%%%%%%%%%%%%%%%%%%%%%%%%%%%%%%%%%%%%%%%%%%%%%%%%%%%%%%%

	\medskip
	%AMS classifications available at https://mathscinet.ams.org/mathscinet/freetools/msc-search
	\noindent MSC2020: 11B39, 33C05
	
\end{document}